\documentclass[reqno]{amsart}

\usepackage[a4paper]{geometry}
\usepackage[utf8]{inputenc}
\usepackage[english]{babel}
\usepackage{amsmath}
\usepackage{amsfonts}
\usepackage{amssymb}
\usepackage{amsthm}
\usepackage{wasysym}
\usepackage{amscdx}
\usepackage{graphicx}
\usepackage{comment}
\usepackage{tikz-cd}
\usepackage{xcolor}
\usepackage{marvosym}
\usepackage{enumitem}
\usepackage{stix}

\usepackage{hyperref}
\hypersetup{colorlinks=true}
\makeatletter
\renewcommand\@mkboth[2]{\markboth{#1}{}}
\makeatother

\theoremstyle{definition}
\newtheorem{thm}{Theorem}[section]
\newtheorem{lem}[thm]{Lemma}
\newtheorem{prop}[thm]{Proposition}
\newtheorem{defn}[thm]{Definition}

\newtheorem{question}[thm]{Question}

\newtheorem*{claim*}{Claim}
\newtheorem{cor}[thm]{Corollary}

\DeclareMathOperator{\dom}{dom}
\DeclareMathOperator{\lh}{lh}

\DeclareMathOperator{\ZFC}{ZFC}
\DeclareMathOperator{\AD}{AD}

\DeclareMathOperator{\Clb}{C(\vec{\lambda})}

\newcommand{\FF}{\mathcal{F}}
\newcommand{\NN}{\mathcal{N}}
\newcommand{\OO}{\mathcal{O}}
\newcommand{\UU}{\mathcal{U}}
\newcommand{\VV}{\mathcal{V}}
\newcommand{\WW}{\mathcal{W}}
\newcommand{\HOD}{HOD}

\renewcommand{\P}{{\mathbb{P}_{\vec{\mathcal{U}}}}}

\newcommand{\lb}{\lambda}
\renewcommand{\o}{\omega}

\newcommand{\R}{\mathbb{R}} 

\newcommand{\rest}{\!\restriction\!}

\newcommand{\Set}[2]{\{{#1}~\vert~{#2}\}}
\newcommand{\map}[3]{{#1}:{#2}\longrightarrow{#3}}
\newcommand{\Map}[5]{{#1}:{#2}\longrightarrow{#3};~{#4}\longmapsto{#5}}
\newcommand{\anf}[1]{{\text{``}\hspace{0.3ex}{#1}\hspace{0.3ex}\text{''}}}
\newcommand{\seq}[2]{\langle{#1}~\vert~{#2}\rangle}

\newcommand{\ran}[1]{{{\rm{ran}}(#1)}}
\newcommand{\crit}[1]{{{\rm{crit}}({#1})}}
\newcommand{\cof}[1]{{{\rm{cof}}({#1})}}
\newcommand{\Ord}{{\rm{Ord}}}
\newcommand{\Lim}{{\rm{Lim}}}
\newcommand{\POT}[1]{{\mathcal{P}}({#1})}
\newcommand{\HHrm}[1]{{H}_{{#1}}}
\newcommand{\goedel}[2]{{\prec}{#1},{#2}{\succ}}
\newcommand{\pre}[2]{{}^{#1}{#2}}

\newenvironment{enumerate-(a)}{\begin{enumerate}[label={\upshape (\alph*)}, leftmargin=2pc]}{\end{enumerate}}
\newenvironment{enumerate-(a)-r}{\begin{enumerate}[label={\upshape (\alph*)}, leftmargin=2pc,resume]}{\end{enumerate}}
\newenvironment{enumerate-(a)-5}{\begin{enumerate}[label={\upshape (\alph*)}, leftmargin=2pc,start=5]}{\end{enumerate}}
\newenvironment{enumerate-(A)}{\begin{enumerate}[label={\upshape (\Alph*)}, leftmargin=2pc]}{\end{enumerate}}
\newenvironment{enumerate-(A)-r}{\begin{enumerate}[label={\upshape (\Alph*)}, leftmargin=2pc,resume]}{\end{enumerate}}
\newenvironment{enumerate-(i)}{\begin{enumerate}[label={\upshape (\roman*)}, leftmargin=2pc]}{\end{enumerate}}
\newenvironment{enumerate-(i)-r}{\begin{enumerate}[label={\upshape (\roman*)}, leftmargin=2pc,resume]}{\end{enumerate}}
\newenvironment{enumerate-(I)}{\begin{enumerate}[label={\upshape (\Roman*)}, leftmargin=2pc]}{\end{enumerate}}
\newenvironment{enumerate-(I)-r}{\begin{enumerate}[label={\upshape (\Roman*)}, leftmargin=2pc,resume]}{\end{enumerate}}
\newenvironment{enumerate-(1)}{\begin{enumerate}[label={\upshape (\arabic*)}, leftmargin=2pc]}{\end{enumerate}}
\newenvironment{enumerate-(1)-r}{\begin{enumerate}[label={\upshape (\arabic*)}, leftmargin=2pc,resume]}{\end{enumerate}}

\newenvironment{enumerate-(star)}{\begin{enumerate}[label={\upshape{(\( \star_{ \arabic*} \))}}, leftmargin=2pc]}{\end{enumerate}}

\makeatletter
\@namedef{subjclassname@2020}{%
  \textup{2020} Mathematics Subject Classification}
\makeatother

\title{Descriptive properties of I2-embeddings}

\author{Vincenzo Dimonte}
\address{Dipartimento di Scienze Matematiche, Informatiche e Fisiche, Universit\`{a} degli Studi di Udine, via delle Scienze, 206, 33100 Udine, Italy}
\email{vincenzo.dimonte@uniud.it}

\author{Martina Iannella}
\address{Institute of Discrete Mathematics and Geometry, Vienna University of Technology, 1040 Vienna, Austria}
\email{martina.iannella@tuwien.ac.at}

\author{Philipp L\"{u}cke}
\address{Fachbereich Mathematik, Universit\"at Hamburg, Bundesstra{\ss}e 55, Hamburg, 20146, Germany}
\email{philipp.luecke@uni-hamburg.de}


\begin{document}

\begin{abstract}
 We contribute to the study of generalizations of the Perfect Set Property and the Baire Property to subsets of spaces of higher cardinalities, 
 like the power set $\POT{\lambda}$ of a singular cardinal $\lambda$ of countable cofinality or products $\prod_{i<\omega}\lambda_i$ for a strictly increasing sequence $\seq{\lambda_i}{i<\omega}$ of cardinals. 
  We consider the question under which large cardinal hypotheses  classes of definable subsets of these spaces possess  such regularity properties, focusing on rank-into-rank axioms and classes of sets definable by $\Sigma_1$-formulas with parameters from  various collections of sets. 
  We prove that $\omega$-many measurable cardinals, while sufficient to prove the perfect set property of all $\Sigma_1$-definable sets with parameters in $V_\lambda\cup\{V_\lambda\}$, are not enough to prove it if there is a cofinal sequence in $\lambda$ in the parameters.  
   For this conclusion, the existence of an I2-embedding is enough, but there are parameters in $V_{\lambda+1}$ for which I2 is still not enough. The situation is similar for the Baire property: under I2 all sets that are $\Sigma_1$-definable using elements of $V_\lambda$ and a cofinal sequence as parameters have the Baire property, but I2 is not enough for some parameter in $V_{\lambda+1}$. Finally, the existence of an I0-embedding implies that all sets that are $\Sigma^1_n$-definable  with parameters in $V_{\lambda+1}$ have the Baire property.
\end{abstract}

\maketitle



\section{Introduction}

 Fundamental results of descriptive set theory show that simply definable sets of real numbers, {e.g.} Borel sets, possess a rich and canonical structure theory and these structural results have various  applications in other areas of mathematics. Moreover, seminal results show that canonical extensions of the axioms of $\ZFC$ allow us to extend  these structural conclusions  to much larger classes of definable sets of reals. 
 Since the developed theory is limited to the study of mathematical objects that can be encoded as definable sets of real numbers, there has been a recent interest to develop a \emph{generalized descriptive set theory} that allows the study of definable objects of much higher cardinalities. 
 While it is already known that several key results of the classical theory cannot be directly generalized to all higher cardinalities (see, for example, \cite{LS15}), the research conducted so far in this area isolated several settings in which rich structure theories for definable sets of higher cardinalities can be developed. 
 The work presented in this paper contributes to the study of one of these settings that originates in Hugh Woodin`s work on large cardinal assumptions close to the \emph{Kunen Inconsistency} (see \cite{Woo11}).

 Remember that a non-trivial elementary embedding $\map{j}{L(V_{\lambda+1})}{L(V_{\lambda+1})}$ for some ordinal $\lambda$ is an \emph{I0-embedding} if $\crit{j}<\lambda$ holds. 
 Kunen's analysis of elementary embeddings in \cite{Ku71} then directly shows that $\lambda=\sup_{n<\omega}\lambda_n$ holds for every I0-embedding $\map{j}{L(V_{\lambda+1})}{L(V_{\lambda+1})}$ with critical sequence\footnote{We say that a sequence $\seq{\lambda_n}{n<\omega}$ of ordinals is the \emph{critical sequence} of a non-trivial elementary embedding $\map{j}{M}{N}$ between transitive classes if $\lambda_0=\crit{j}$ and $j(\lambda_n)=\lambda_{n+1}$ holds for all $n<\omega$.} $\seq{\lambda_n}{n<\omega}$. 
 Embeddings of this type produce a setting in which descriptive concepts can be developed fruitfully. More specifically,  several deep results show that the structural properties of the collection of subsets of $\POT{\lambda}$ contained in $L(V_{\lambda+1})$ strongly resembles the behavior of the collection of sets of reals in $L(\R)$  in the presence of the \emph{Axiom of Determinacy  $\AD$} in $L(\R)$. 
 In the following, we will focus on generalizations of the \emph{Perfect Set Property} to definable subsets of higher power sets. 
 Given a non-empty set $X$ and an infinite cardinal $\mu$, we equip the set ${}^\mu X$ of all functions from $\mu$ to $X$ with the topology whose basic open sets consists of all functions that extend a given function $\map{s}{\xi}{X}$ with $\xi<\mu$. 
In addition, we equip the  set $\POT{\nu}$ of all subsets of an infinite cardinal $\nu$ with the topology whose basic open sets consist of all subsets of $\nu$ whose intersection with a given ordinal $\eta<\nu$ is equal to a fixed subset of $\eta$. 
Finally, we say that a  map $\map{\iota}{X}{Y}$ between topological spaces is a \emph{perfect embedding} if it induces a homeomorphism between $X$ and the subspace $\ran{\iota}$ of $Y$.

 It is easy to see that  for every infinite cardinal $\lambda$, there is a subset of $\POT{\lambda}$ of cardinality greater than $\lambda$ that does not contain the range of a perfect embedding of ${}^{\cof{\lambda}}\lambda$ into $\POT{\lambda}$.\footnote{First, observe that for every $\gamma<\lambda$, the set $\POT{\gamma}$ is discrete in $\POT{\lambda}$ and therefore it does not contain the range of a perfect embedding of ${}^{\cof{\lambda}}\lambda$ into $\POT{\lambda}$. In particular, if $2^{{<}\lambda}>\lambda$, then there is a subset of $\POT{\lambda}$ with the desired property. In the other case, if $2^{{<}\lambda}=\lambda$, then the set of perfect embeddings of ${}^{\cof{\lambda}}\lambda$ into $\POT{\lambda}$ has cardinality $2^\lambda$ and we can build the desired subset through a standard recursive construction.} 
 In contrast, classical results show that if  $\AD$ holds in $L(\R)$, then every uncountable subset of $\POT{\omega}$ in $L(\R)$ contains the range of a perfect embedding of ${}^\omega\omega$ into $\POT{\omega}$. 
 The work of Hugh Woodin, Xianghui Shi and Scott Cramer now shows that I0-embeddings imply an analogous dichotomy at the supremum of the corresponding critical sequence (see {\cite[Section 5]{Cra15}}, {\cite[Section 4]{Sh15}} and {\cite[Section 7]{Woo11}}).

 \begin{thm}[{\cite{Cra15}}]\label{thm:PerfectI0}
  If $\map{j}{L(V_{\lambda+1})}{L(V_{\lambda+1})}$ is an I0-embedding and $X$ is a subset of $\POT{\lambda}$ of cardinality greater than $\lambda$ that is an element of $L(V_{\lambda+1})$, then there is a perfect embedding $\map{\iota}{{}^\omega\lambda}{\POT{\lambda}}$ with $\ran{\iota}\subseteq X$. 
 \end{thm}

 The work presented in this paper is  motivated by the question whether the restriction of this implication to smaller classes of definable sets can be derived from weaker large cardinal assumptions. 
 It is motivated by the results of Sandra M\"uller and the third author in \cite{LM21} that analyze simply definable sets at limits of measurable cardinals. 
 In the following, we say a class $C$ is \emph{definable by a formula $\varphi(v_0,\ldots,v_n)$ and parameters $z_0,\ldots,z_{n-1}$} if $C=\Set{y}{\varphi(y,z_0,\ldots,z_{n-1})}$ holds.
 We now distinguish classes of definable sets using the \emph{Levy hierarchy of formulas}\footnote{See {\cite[p. 5]{MR1994835}}.} and the rank of parameters. 
 The following result is the starting point of our work:

 \begin{thm}[\cite{LM21}]\label{thm:LM}
  If $\lambda$ is a limit of measurable cardinals and $X$ is a subset of $\POT{\lambda}$ of cardinality greater than $\lambda$ that is definable by a $\Sigma_1$-formula with parameters in $V_\lambda\cup\{\lambda\}$, then there is a perfect embedding $\map{\iota}{{}^{\cof{\lambda}}\lambda}{\POT{\lambda}}$ with $\ran{\iota}\subseteq X$. 
 \end{thm}

 Given an infinite cardinal $\lambda$, the \emph{$\Sigma_1$-Reflection Principle} shows that all $\Sigma_1$-formulas with parameters in $\HHrm{\lambda^+}$ are absolute between $V$ and $\HHrm{\lambda^+}$. Therefore, it follows that a subset of $\HHrm{\lambda^+}$ is definable by a $\Sigma_1$-formula with parameters in $\HHrm{\lambda^+}$ if and only if the given set is definable in this way in  $\HHrm{\lambda^+}$.
  This shows that,  if $\lambda$ is an infinite cardinal with $\HHrm{\lambda}=V_\lambda$, then $L(V_{\lambda+1})$ contains all subsets of $\POT{\lambda}$ that are definable by a $\Sigma_1$-formula with parameters in $\HHrm{\lambda^+}$, because $\HHrm{\lambda^+}$ is contained in $L(V_{\lambda+1})$.  
  In particular, it follows that  the conclusion of the implication stated in Theorem \ref{thm:PerfectI0} directly implies the conclusion of the implication stated in Theorem \ref{thm:LM}.

The theorems cited above directly raise the question if stronger perfect set theorems can be proven for limits of countably many measurable cardinals. In particular, it is natural to ask if the implication of Theorem \ref{thm:LM} still holds true if we allow more elements of $\POT{\lambda}$ in our $\Sigma_1$-definitions. 
 A natural candidate for such an additional parameter in $\POT{\lambda}\setminus(V_\lambda\cup\{\lambda\})$ is an $\omega$-sequence of measurable cardinals that is cofinal in the given supremum $\lambda$. 
 Our first result, proven in Section \ref{section:Negative},  shows that we no longer get a provable implication if we are allowed to use such a sequence as a parameter in our $\Sigma_1$-definitions:

 \begin{thm}\label{thm:NegativeLimitMeasurable}
	If $\vec{\lambda}$ is a strictly increasing sequence of measurable cardinals with limit $\lambda$, then the following statements hold in an inner model $M$ containing $\vec{\lambda}$:  
	\begin{enumerate-(i)}
	 \item The sequence $\vec{\lambda}$ consists of measurable cardinals. 
	 
	 \item If $\vec{\nu}$ is a strictly increasing $\omega$-sequence of  regular cardinals  with limit $\lambda$, then there is a subset of $\POT{\lambda}$ of cardinality greater than $\lambda$ that does not contain the range of a perfect embedding of ${}^\omega\lambda$ into $\POT{\lambda}$ and is definable  by a $\Sigma_1$-formula with parameters in $V_\lambda\cup\{\vec{\nu}\}$.  
	\end{enumerate-(i)} 
\end{thm}

 We now proceed by showing that a large cardinal axiom strictly weaker than the existence of an I0-embedding implies the perfect set property discussed above. 
 Remember that an elementary embedding $\map{j}{V}{M}$ with critical sequence $\seq{\lambda_n}{n<\omega}$ is an  \emph{I2-embedding} if $V_\lambda\subseteq M$, where $\lambda=\sup_{n<\omega}\lambda_n$. The existence of such an embedding is equivalent to the existence of  a non-trivial elementary embedding $\map{i}{V_\lambda}{V_\lambda}$ with critical sequence $\seq{\nu_n}{n<\omega}$ such that $\lambda=\sup_{n<\omega}\nu_n$ and the canonical map $$\Map{i_+}{V_{\lambda+1}}{V_{\lambda+1}}{A}{\bigcup\Set{i(A\cap V_{\lambda_n})}{n<\omega}}$$ extending $i$ to $V_{\lambda+1}$ maps well-founded relations on $V_\lambda$ to well-founded relations on $V_\lambda$ (see {\cite[Proposition 24.2]{MR1994835}}). 
 The results of \cite{La97} show that, if $\map{i}{L(V_{\nu+1})}{L(V_{\nu+1})}$ is an I0-embedding, then there is an embedding $\map{j}{V_\lambda}{V_\lambda}$ for some $\lambda<\nu$ with the given property. Since  $\nu$ is a limit of inaccessible cardinals in this setting, it follows that the existence of an I0-embedding has strictly higher consistency strength than the existence of an I2-embedding. 
 The next result, proven in Section \ref{section:Positive}, shows that I2-embeddings imply the desired perfect set property:

\begin{thm}\label{thm:Main1}
 Let $\map{j}{V}{M}$ be an I2-embedding with critical sequence $\vec{\lambda}=\seq{\lambda_n}{n<\omega}$ and set $\lambda=\sup_{n<\omega}{\lambda_n}$.  If $X$ is a subset of $\POT{\lambda}$ of cardinality greater than $\lambda$ that is definable by a $\Sigma_1$-formula with parameters in $V_\lambda\cup\{V_\lambda,\vec{\lambda}\}$, then there is a perfect embedding $\map{\iota}{{}^\omega\lambda}{\POT{\lambda}}$ with $\ran{\iota}\subseteq X$. 
\end{thm}

 The proof of this theorem will  show that its conclusion holds for subsets of $\POT{\lambda}$ that are definable from a significantly larger set of parameters in $V_{\lambda+1}$ (see Theorem \ref{thm:Main1PlusParameters} below). 
 However, in Section \ref{section:Negative}, we will observe that an assumption strictly stronger than the existence of an I2-embedding is necessary to obtain this perfect set property for all subsets of $\POT{\lambda}$ that are definable by $\Sigma_1$-formulas with parameters in $\POT{\lambda}$:

\begin{thm}\label{thm:NoI2boldface}
 If $\map{j}{V}{M}$ is an I2-embedding with critical sequence $\vec{\lambda}=\seq{\lambda_n}{n<\omega}$ and $\lambda=\sup_{n<\omega}\lambda_n$, then the following statements hold in an inner model:   
 \begin{enumerate-(i)}
  \item There is an I2-embedding whose critical sequence has supremum $\lambda$. 
  
  \item There is a subset $z$ of $\lambda$ and a subset $X$ of $\POT{\lambda}$ of cardinality greater than $\lambda$ such that $X$ does not contain the range of a perfect embedding of ${}^\omega\lambda$ into $\POT{\lambda}$  and the set $X$ is definable by a $\Sigma_1$-formula with parameter $z$. 
  \end{enumerate-(i)}
  \end{thm}

 The five results discussed above suggest the intriguing possibility of studying large cardinal assumptions canonically inducing singular cardinals $\lambda$ of countable cofinality through the provable regularity properties of simply definable subsets of $\POT{\lambda}$. 
 More specifically, they suggest that for each large cardinal axiom of this form, we want to uniformly assign  as large subsets $P_\lambda$ of $V_{\lambda+1}$ as possible to each singular cardinal $\lambda$, in a way that ensures that $\ZFC$ proves that whenever $\lambda$ is a singular cardinal of  countable cofinality induced by a cardinal of the given type,  then all subsets of $\POT{\lambda}$ that are definable by $\Sigma_1$-formulas using parameters from $P_\lambda$ either have cardinality at most $\lambda$ or  contain the range of a perfect embedding of ${}^\omega\lambda$ into $\POT{\lambda}$. 
 Note that, since this approach is based on provable implications and not consistency strength, it is less affected by the current technical limitations of inner model theory and therefore provides a new angle to study strong large cardinal axioms.

 In addition to $\Sigma_1$-definable subsets of power sets, we will also study spaces and complexity classes that more closely resemble the objects studied in classical descriptive set theory. 
 More specifically, for a given strictly increasing sequence $\vec{\lambda}=\seq{\lambda_n}{n<\omega}$ of infinite cardinals with supremum $\lambda$, we will study subsets of the  closed subspace $C(\vec{\lambda})$ of ${}^\omega\lambda$ consisting of all functions in the set $\prod_{n<\omega}\lambda_n$, {i.e.,} all functions $\map{x}{\omega}{\lambda}$ satisfying $x(n)<\lambda_n$ for all $n<\omega$. 
Note that the map $$\Map{\iota_{\vec{\lambda}}}{C(\vec{\lambda})}{\POT{\lambda}}{x}{\Set{\goedel{\lambda_n}{x(n)}}{n<\omega}}$$ yields a homeomorphism between $C(\vec{\lambda})$ and a closed subset of $\POT{\lambda}$.\footnote{Here, we let $\map{\goedel{\cdot}{\cdot}}{\Ord\times\Ord}{\Ord}$ denote the \emph{G\"odel pairing function}.} 
Moreover, since the map $\iota_{\vec{\lambda}}$ is definable by a $\Delta_0$-formula with parameter $\vec{\lambda}$, Theorem \ref{thm:Main1} immediately implies a perfect set theorem for subsets of $C(\vec{\lambda})$ definable by $\Sigma_1$-formulas with parameters in the set $V_\lambda\cup\{\vec{\lambda}\}$. 
 Finally, the sets produced in the proofs of Theorems \ref{thm:NegativeLimitMeasurable} and  \ref{thm:NoI2boldface} will actually be subsets of $\ran{\iota_{\vec{\lambda}}}$ and therefore yield analogous negative results for $\Sigma_1$-definable subsets of $C(\vec{\lambda})$ (see Theorems \ref{theorem:NegativeI2Clambda} and \ref{thm:NegativeLimitMeasurableC-lambda} below). 

  The theorems above extends beyond $\POT{\lambda}$ and $C(\vec{\lambda})$: In \cite{DMR23} a whole classes of spaces is introduced: the $\lambda$-Polish spaces, {i.e.,} spaces that are completely metrizable and with weight $\lambda$, and it is easy to prove analogous results for them. For example, $\pre{\lambda}{2}$, with the bounded topology, is homeomorphic to $\POT{\lambda}$, and therefore Theorems \ref{thm:LM}, \ref{thm:NegativeLimitMeasurable}, \ref{thm:Main1} and \ref{thm:NoI2boldface} hold in there. The space $\pre{\omega}{\lambda}$, with the product topology, is homeomorphic to a closed subset of $\POT{\lambda}$ via the map $x\mapsto \goedel{n}{x(n)}$, and it contains $C(\vec{\lambda})$ as a closed set, therefore Theorems \ref{thm:NegativeLimitMeasurable}, \ref{thm:Main1} and \ref{thm:NoI2boldface} hold in there. If $(X,d)$ is any $\lambda$-Polish space, then there is a $\Sigma_1(d)$ continuous bijection between a closed set $F\subseteq\pre{\lambda}{\omega}$ and $X$ (\cite{DMR23}). By pulling back with the bijection, we can therefore prove Theorems \ref{thm:LM}, \ref{thm:NegativeLimitMeasurable}, \ref{thm:Main1} and \ref{thm:NoI2boldface} also in there. Finally, if $d$ is more complicated, the negative results of Theorems \ref{thm:NegativeLimitMeasurable} and \ref{thm:NoI2boldface} hold anyway, but respectively with a witness in $\Sigma_1(V_{\lambda}\cup\{V_\lambda,\vec{\lambda},d\})$ and in $\Sigma_1(z,d)$.
	
  In another direction, we will not only study subsets of $\POT{\lambda}$, ${}^\omega \lambda$ or $C(\vec{\lambda})$ that are definable in $V$ by formulas of a given complexity, but also sets that are definable over $V_\lambda$ by higher-order formulas in the classes $\Sigma^m_n$ and $\Pi^m_n$ (see, for example, {\cite[p. 295]{J03}}) using certain parameters contained in $V_{\lambda+1}$. The following results (whose proof is a routine adaptation of the proof of {\cite[ Lemma 25.25]{J03}} to higher cardinals of countable cofinality) connects this form of definability to $\Sigma_1$-definitions:

\begin{prop}\label{prop:translate}
  For every $\Sigma_1$-formula  $\varphi(v_0,\ldots,v_{k-1})$ in the language of set theory, there exists a $\Sigma^1_2$-formula $\psi(w_0,\ldots,w_{k-1})$ in the language of set theory with free second-order parameters $w_0,\ldots,w_{k-1}$ such that $\ZFC$ proves that $$\varphi(A_0,\ldots,A_{k-1}) ~ \Longleftrightarrow ~ \langle V_\lambda,\in\rangle\models\psi(A_0,\ldots,A_{k-1})$$ holds  for every singular cardinal $\lambda$ of countable cofinality with $\HHrm{\lambda}=V_\lambda$ and all $A_0,\ldots,A_{k-1}\in V_{\lambda+1}$.  
\end{prop}

We will later show (see Corollary \ref{corollary:TranslateBack}) that, in certain contexts, it is also possible to translate $\Sigma^1_2$-formulas into $\Sigma_1$-formulas. 
 Moreover, note that, in \cite{DMR23},  Luca Motto Ros and the first author prove that, analogous to the classical setting, for every singular strong limit cardinal $\lambda$ of countable cofinality, every $\mathbf{\Sigma}^1_1$-subset of ${}^\omega\lambda$ ({i.e.,} every subset of ${}^\omega\lambda$ that is definable over $V_\lambda$ by a $\Sigma_1^1$-formula with parameters in $V_{\lambda+1}$) of cardinality greater than $\lambda$ contains the range of a perfect embedding of ${}^\omega\lambda$ into itself. In addition, still completely analogous to the classical setting, they show that, if $V=L$ holds and $\lambda$ is a singular cardinal of countable cofinality, then there is a subset of ${}^\omega\lambda$ of cardinality $\lambda^+$ that is definable over $V_\lambda$ by a $\Pi^1_1$-formula without parameters.

In addition, we later will consider an analog of the Baire Property to $\lambda$, that we call \emph{$\vec{\mathcal{U}}$-Baire property} (see Definition \ref{definition:U-BP} below). 
 In analogy with Theorem \ref{thm:Main1}, the existence of an I2-embedding with supremum of the critical sequence $\lambda$ implies that every subset of \(\Clb\) that is definable by a $\Sigma_1$-formula with parameters in $V_\lambda\cup\{V_\lambda,\vec{\lambda}\}$  has the $\vec{\mathcal{U}}$-Baire property (see Theorem \ref{thm:BP for lightface} below).
 Moreover, in analogy with Theorem \ref{thm:PerfectI0}, the existence of  an I0-embedding with supremum of the critical sequence $\lambda$ implies that every subset of \(\Clb\) in $L_1(V_{\lambda+1})$ has the $\vec{\mathcal{U}}$-Baire property (see Theorem \ref{thm:I0-lightface-BP} below). 
  Finally, as a negative result, we show that, in the inner model constructed in the proof of Theorem \ref{thm:NoI2boldface}, there exists an I2-embedding with supremum of the critical sequence $\lambda$ and a subset of  \(\Clb\) without the $\vec{\mathcal{U}}$-Baire property that is definable by a $\Sigma_1$-formula with parameters in $V_{\lambda+1}$ (see Theorem \ref{theorem:Sigma1No-UBP} below).


\section{Negative results}\label{section:Negative}

In this section, we will prove the restricting results stated in the introduction (Theorems \ref{thm:NegativeLimitMeasurable} and \ref{thm:NoI2boldface}). 
Theorem \ref{thm:NegativeLimitMeasurable} motivates the formulation of the main result of this paper (Theorem \ref{thm:Main1}) by showing that its conclusion cannot be derived from the weaker  large cardinal assumptions used in Theorem \ref{thm:LM}. 
On the other hand, Theorem \ref{thm:NoI2boldface} shows that the statement of Theorem \ref{thm:Main1}  cannot be strengthened to  affect all sets  that are $\Sigma_1$-definable from arbitrary subsets of the given singular cardinal. 
 In the following, we use  arguments  based on ideas and notions that were already used in {\cite[Section 4]{LM21}}.

\begin{defn}
 Let $\vec{\lambda}=\seq{\lambda_n}{n<\omega}$ be a strictly increasing sequence of cardinals with supremum $\lambda$  and let $\vec{a}=\seq{a_\alpha}{\alpha<\lambda}$ be a  sequence of elements of $V_{\lambda}$. 
 \begin{enumerate-(i)}
     \item Given $x\subseteq\lambda$, we define $\lhd_x$ to be the unique binary relation on $\lambda$ with the property that $$\alpha\lhd_x\beta ~ \Longleftrightarrow ~ \goedel{\alpha}{\beta}\in x$$ holds for all $\alpha,\beta<\lambda$.  
     
     \item We define $\WW\OO_\lambda$ to be the set of all $x\in\POT{\lambda}$ with the property that $\lhd_x$ is a well-ordering of $\lambda$. 
     
    \item We let $\mathsf{WO}(\vec{\lambda},\vec{a})$ denote the set of all $b\in{}^\omega\lambda$ with the property that there exists $x\in\WW\OO_\lambda$ such that  $x\cap\lambda_n=a_{b(n)}$ holds for all $n<\omega$. 
     
    \item Given an element  $b$ of $\mathsf{WO}(\vec{\lambda},\vec{a})$, we let $\|b\|_{\vec{a}}$ denote the order-type of the resulting well-order  $\langle\lambda,\lhd_{\hspace{0.9pt}\bigcup \Set{a_{b(n)}}{n<\omega}}\rangle$. 
 \end{enumerate-(i)}
\end{defn}

The following \emph{Boundedness Lemma} now follows from the theory developed in {\cite[Section 4]{LM21}} that generalizes classical arguments from descriptive set theory to singular strong limit cardinals of countable cofinality:

\begin{lem}[{\cite[Lemma 4.5]{LM21}}]\label{lemma:BoundednessLemma+}
 Let $\vec{\lambda}=\seq{\lambda_n}{n<\omega}$ be a strictly increasing sequence of inaccessible cardinals with supremum $\lambda$  and let $\vec{a}=\seq{a_\alpha}{\alpha<\lambda}$ be an enumeration  of $\HHrm{\lambda}$. 
 If $\map{f}{{}^\omega\lambda}{{}^\omega\lambda}$ is a continuous function with $\ran{f}\subseteq\mathsf{WO}(\vec{\lambda},\vec{a})$, 
 then there exists an ordinal $\gamma<\lambda^+$ with $\|f(c)\|_{\vec{a}}<\gamma$ for all $c\in{}^\omega\lambda$. 
 %
\end{lem}

We start by limiting the provable structural consequences of I2-embeddings by proving the following strengthening of Theorem  \ref{thm:NoI2boldface} that shows that the statement of Theorem \ref{thm:Main1} cannot be strengthened to show that the existence of an I2-embedding at a cardinal $\lambda$ implies that every subset of $\POT{\lambda}$ that is definable by a $\Sigma_1$-formula with parameters in $V_{\lambda+1}$ either has cardinality $\lambda$ or contains the range of a perfect embedding:

\begin{thm}\label{theorem:NegativeI2Clambda}
 If $\map{j}{V}{M}$ is an I2-embedding with critical sequence $\vec{\lambda}=\seq{\lambda_n}{n<\omega}$ and $\lambda=\sup_{n<\omega}\lambda_n$, then the following statements hold in an inner model:   
 \begin{enumerate-(i)}
  \item There is an I2-embedding whose critical sequence has supremum $\lambda$. 
  
  \item    There is a subset $z$ of $\lambda$ and a subset $X$ of $C(\vec{\lambda})$ of cardinality greater than $\lambda$ such that $X$ does not contain the range of a perfect embedding of ${}^\omega\lambda$ into $C(\vec{\lambda})$  and the set $X$ is definable by a $\Sigma_1$-formula with parameter $z$. 
  \end{enumerate-(i)}
  \end{thm}
  
\begin{proof}
 Since $\lambda$ is a limit of inaccessible cardinals, we can find a subset $y$ of $\lambda$ with the property that $V_\lambda\cup\{\vec{\lambda}, ~  j\restriction V_\lambda\}\subseteq L[y]$. 
 Since this setup ensures that $$(j\restriction V_\lambda)_+^{L[y]} ~ = ~  (j\restriction V_\lambda)_+\restriction V_{\lambda+1}^{L[y]},$$ 
 we know that $(j\restriction V_\lambda)_+^{L[y]}$  maps well-founded relations on $V_\lambda$ in $L[y]$ to well-founded relations on $V_\lambda$ in $L[y]$ and 
 it follows that $j\restriction V_\lambda$ witnesses that, in $L[y]$, there is an I2-embedding whose critical sequence has supremum $\lambda$.

 Now, work in $L[y]$. First, observe that the set $\WW\OO_\lambda$ consists of all subsets $x$ of $\lambda$ with the property that there exists an ordinal $\gamma$ and an order isomorphism between $\langle\lambda,\lhd_x\rangle$ and $\langle\gamma,<\rangle$. In addition, the set $\WW\OO_\lambda$ also consists of all subsets $x$ of $\lambda$ such that  $\langle\lambda,\lhd_x\rangle$ is a linear ordering with the property that no injective sequence $\seq{\alpha_n}{n<\omega}$ is decreasing with respect to $\lhd_x$. This shows that  $\WW\OO_\lambda$   is $\Delta_1$-definable\footnote{Given a natural number $n>0$, a class $C$ is \emph{$\Delta_n$-definable} from a parameter $p$ if the classes $C$ and $V\setminus C$ are both definable by $\Sigma_n$-formulas with parameter $p$.} from the parameter $\lambda$. 
 Pick  an enumeration $\vec{a}=\seq{a_\alpha}{\alpha<\lambda}$ of $V_\lambda$ with $V_{\lambda_n}=\Set{a_\alpha}{\alpha<\lambda_n}$  for all $n<\omega$. 
 Then there exists an unbounded subset $z$ of $\lambda$ with the property that the sets $\{\vec{a}\}$, $\{y\}$ and $\{\vec{\lambda}\}$ are all definable by $\Sigma_1$-formulas with parameter $z$. Note that this implies that these sets are actually $\Delta_1$-definable from the parameter $z$. 
 Note that an element $b$ of ${}^\omega\lambda$ is not contained in $\mathsf{WO}(\vec{\lambda},\vec{a})$ if and only if either there are $m<n<\omega$ with $a_{b(n)}\cap\lambda_m\neq a_{b(m)}$ or there exists $x\in\POT{\lambda}\setminus\WW\OO_\lambda$ with   $x\cap\lambda_n=a_{b(n)}$ holds for all $n<\omega$. Together with our earlier observations, this shows that the set $\mathsf{WO}(\vec{\lambda},\vec{a})$ is $\Delta_1$-definable from the parameter $z$. 
 Given $\lambda\leq\gamma<\lambda^+$, we now let $b_\gamma$ denote the $<_{L[y]}$-least element of $\mathsf{WO}(\vec{\lambda},\vec{a})$ with $\|b_\gamma\|_{\vec{a}}=\gamma$ and $b_\gamma(n)<\lambda_{n+1}$  for all $n<\omega$. 
  Note that our setup ensures that such a set exists for all    $\lambda\leq\gamma<\lambda^+$. 
 Moreover, since the basic structure theory of $L[y]$ ensures that the class of proper initial segments of $<_{L[y]}$ is definable by a $\Sigma_1$-formula with parameter $z$, the fact that $\mathsf{WO}(\vec{\lambda},\vec{a})$ is $\Delta_1$-definable from the parameter $z$ yields a $\Sigma_1$-formula $\varphi(v_0,v_1,v_2)$ with the property that $\varphi(\gamma,b,z)$ holds if and only if $\gamma$ is an ordinal in the interval $[\lambda,\lambda^+)$ and $b=b_\gamma$. 
 Let $X$ denote the set of all $b\in{}^\omega\lambda$ with the property that $b(0)=0$ and there exists $\lambda\leq\gamma<\lambda^+$ with $b(n+1)=b_\gamma(n)$ for all $n<\omega$. We then know that $X$ is a subset of $C(\vec{\lambda})$ of cardinality greater than $\lambda$ that  is definable by a $\Sigma_1$-formula with parameter $z$.

 Assume, towards a contradiction, that there is a perfect embedding $\map{\iota}{{}^\omega\lambda}{C(\vec{\lambda})}$ with $\ran{\iota}\subseteq X$. Set $Y=\Set{b_\gamma}{\lambda\leq\gamma<\lambda^+}$ and let $\map{\zeta}{X}{Y}$ denote the unique map with $\zeta(b)(n)=b(n+1)$ for all $b\in X$ and $n<\omega$. Then $\zeta$ is a homeomorphism of the subspace $X$ of $C(\vec{\lambda})$ and the subspace $Y$ of $C(\vec{\lambda})$. In particular, it follows that $\zeta\circ\iota$ is a perfect embedding of ${}^\omega\lambda$ into $C(\vec{\lambda})$ with $\ran{\zeta\circ\iota}\subseteq Y\subseteq\mathsf{WO}(\vec{\lambda},\vec{a})$. 
 In this situation, Lemma \ref{lemma:BoundednessLemma+} yields $c,d\in{}^\omega\lambda$ with $c\neq d$ and $\|(\zeta\circ\iota)(c)\|_{\vec{a}}=\|(\zeta\circ\iota)(d)\|_{\vec{a}}$. By the definition of $Y$, this is a contradiction.  
\end{proof}

 Note that, in order to construct an inner model $N$ with $V_\lambda\subseteq N$ and the property that (ii) of the above theorem holds, it suffices to assume that $\lambda$ is the supremum of $\omega$-many inaccessible cardinals in order to carry out the construction made in the proof of the theorem. 

 In the remainder of this section, we further develop the arguments used in the above proof to obtain the following strengthening of Theorem \ref{thm:NegativeLimitMeasurable}:

 \begin{thm}\label{thm:NegativeLimitMeasurableC-lambda}
	If $\vec{\lambda}=\seq{\lambda_n}{n<\omega}$ is a strictly increasing sequence of measurable cardinals with limit $\lambda$, then the following statements hold in an inner model $M$ containing $\vec{\lambda}$:  
	\begin{enumerate-(i)}
	 \item The sequence $\vec{\lambda}$ consists of measurable cardinals. 
	 
	 \item If $\vec{\nu}$ is a strictly increasing $\omega$-sequence of  cardinals of uncountable cofinality with limit $\lambda$, then for some $x\in \HHrm{\aleph_1}$,  there is a subset of $C(\vec{\lambda})$ of cardinality greater than $\lambda$ that does not contain the range of a perfect embedding of ${}^\omega\lambda$ into $C(\vec{\lambda})$ and is definable  by a $\Sigma_1$-formula with parameters  $\vec{\nu}$ and $x$.  
	\end{enumerate-(i)} 
\end{thm}

\begin{proof}
 Pick a sequence $\seq{U_n}{n<\omega}$  with the property that $U_n$ is a normal ultrafilter on $\lambda_n$ for all $n<\omega$ and  define $$\UU ~ = \Set{\langle n,A\rangle}{n<\omega, ~ A\in U_n}.$$ Then $\vec{\lambda}\in L[\UU]$ and for every $n<\omega$, the cardinal $\lambda_n$ is measurable in $L[\UU]$. 

 Now, work in $L[\UU]$ and fix a strictly increasing sequence $\vec{\nu}=\seq{\nu_n}{n<\omega}$  of  cardinals of uncountable cofinality with limit $\lambda$. Using standard arguments about iterated measurable ultrapowers (see {\cite[Lemma 19.5]{MR1994835}} and {\cite[Section 3]{MR3411035}}), we can find  
  \begin{itemize}
   \item a transitive class $M$, 
   
   \item an elementary embedding $\map{j}{V}{M}$ with $j(\lambda)=\lambda$, 
   
   \item a function $\map{x}{\omega}{\omega}$, and 
   
   \item a sequence $\seq{C_n}{n<\omega}$ 
  \end{itemize}
  such that the following statements hold for all $n<\omega$: 
  \begin{enumerate-(i)}
   \item\label{thm:lim_meas_card-1} $j(\lambda_n)=\nu_{x(n)}$. 
   
   \item $\nu_{x(n+1)}>\vert{\HHrm{\nu_{x(n)}^+}}\vert$. 
   
   \item $C_n$ is a closed unbounded subset of $\nu_{x(n)}$. 
   
   \item\label{thm:lim_meas_card-4} $j(U_n)=\Set{A\in M\cap\POT{\nu_{x(n)}}}{\exists\xi<\nu_{x(n)} ~ C_n\setminus\xi\subseteq A}$. 
  \end{enumerate-(i)}

 Now, set $\VV=j(\UU)$ and define $\NN$ to be the class of all pairs $\langle N,\vec{F}\rangle$ with the property that $N$ is a  transitive set  of cardinality $\lambda$, $\vec{F}=\seq{F_n}{n<\omega}$ is a sequence of length $\omega$ and  there exists a sequence $\seq{D_n}{n<\omega}$ such that the following statements hold: 
 \begin{enumerate}
  \item[(a)] $D_n$ is a closed unbounded subset of $\nu_{x(n)}$ for all $n<\omega$. 
  
  \item[(b)] If $n<\omega$, then $F_n$ is an element of $N$, $\nu_{x(n)}$ is a regular cardinal in $N$ and $F_n$ is a normal ultrafilter on $\nu_{x(n)}$ in $N$. 
  
  \item[(c)]  If $n<\omega$, then $F_n=\Set{A\in N\cap\POT{\nu_{x(n)}}}{\exists\xi<\nu_{x(n)} ~D_n\setminus\xi\subseteq A}$. 
  
  \item[(d)] If $\FF  = \Set{\langle n,A\rangle}{n<\omega, ~ A\in F_n}$, then $\FF\in N$ and $N=L_{N\cap\Ord}[\FF]$. 
 \end{enumerate}
 It is easy to see that the class $\NN$ is definable by a $\Sigma_1$-formula with parameters $\vec{\nu}$ and $x$. 
 Moreover, our assumptions ensure that for every $x\in M\cap\POT{\lambda}$, there exists $\gamma<\lambda^+$ with $x\in L_\gamma[\VV]$ and $\langle L_\gamma[\VV],\seq{j(U_n)}{n<\omega}\rangle\in \NN$.

 \begin{claim*}
  If $\langle N,\seq{F_n}{n<\omega}\rangle\in\NN$ and $\FF  = \Set{\langle n,A\rangle}{n<\omega, ~ A\in F_n}$, then we have $\FF\cap N=\VV\cap L_{N\cap\Ord}[\VV]$ and $N=L_{N\cap\Ord}[\VV]$. 
 \end{claim*}
 
 \begin{proof}[Proof of the Claim]
  Let $\seq{D_n}{n<\omega}$ be a sequence that witnesses that $\langle N,\seq{F_n}{n<\omega}\rangle$ is contained in $\NN$. 
  Set  $\gamma=N\cap\Ord$. 
  By induction, we now show that $$\FF\cap L_\beta[\FF] ~ = ~ \VV\cap L_\beta[\VV]$$  holds for all $\beta\leq\gamma$.  Hence, assume that $\beta\leq\gamma$ with $\FF\cap L_\alpha[\FF]=\VV\cap L_\alpha[\VV]$ for all $\alpha<\beta$.  Then $L_\beta[\FF]=L_\beta[\VV]$. 
  Pick $n<\omega$ and $A\in F_n$ with $\langle n,A\rangle\in L_\beta[\FF]$. Then there exists $\xi<\nu_{x(n)}$ with $D_n\setminus\xi\subseteq A$. Since $C_n\cap D_n$ is unbounded in $\nu_{x(n)}$, we know that $A\cap C_n$ is unbounded in $\nu_{x(n)}$ and hence there is no $\zeta<\nu_{x(n)}$ with the property that $C_n\setminus\zeta\subseteq \nu_{x(n)}\setminus A$. 
  In this situation, the fact that $j(U_n)$ is an ultrafilter on $\nu_{x(n)}$ in $L[\VV]$ implies that $A\in j(U_n)$ and hence $\langle n,A\rangle\in j(\UU)\cap L_\beta[\VV]=\VV\cap L_\beta[\VV]$. 
  The dual argument then shows that we also have $\VV\cap L_\beta[\VV]\subseteq \FF\cap L_\beta[\FF]$. 
  This completes the induction and we know that $\FF\cap N=\VV\cap L_\gamma[\VV]$. 
  This allows us to conclude that $$N ~ = ~ L_\gamma[\FF] ~ = ~ L_\gamma[\FF\cap N] ~ = ~ L_\gamma[\VV\cap L_\gamma[\VV]] ~ = ~ L_\gamma[\VV],$$ completing the proof of the claim. 
 \end{proof}
 
    Now, note that  (ii) above ensures that there is a sequence $\seq{a_\alpha}{\alpha<\lambda}$ in $M$ with the property that 
    \begin{equation}\label{equation:EnumerateinLF}
      M\cap\POT{\nu_{x(n)}} ~ = ~ \Set{a_\alpha}{\alpha<\nu_{x(n+1)}}
   \end{equation} 
   holds for all $n<\omega$. Define $\vec{a}$ to be the $<_{L[\VV]}$-least sequence in $M$ with this property.

   \begin{claim*}
    The set $\{\vec{a}\}$ is definable by a $\Sigma_1$-formula with parameters $\vec{\nu}$ and $x$. 
   \end{claim*}
   
   \begin{proof}[Proof of the Claim]
    First, note that our previous claim implies that, if $\langle N,\seq{F_n}{n<\omega}\rangle$ is an element of $\NN$ with $\lambda\in N$, then $N=L_{N\cap\Ord}[\VV]$ and $N$ contains all bounded subsets of $\lambda$ in $M$. 
    It follows that $\vec{a}$ is the unique sequence of length $\lambda$ with the property that there exists $\langle N, \seq{F_n}{n < \omega}\rangle$ in $\mathcal{N}$ and $\FF = \Set{\langle n, A\rangle}{n < \omega, A \in F_n}$ such that $\vec{a}$ is the $<_{L[\FF]}$-least element of $N$ with \eqref{equation:EnumerateinLF} for all $n<\omega$. This characterization directly yields the desired $\Sigma_1$-definition.  
   \end{proof}

   Next, notice that, if $y$ is an element of $\WW\OO_\lambda^M$, then $M$ contains an order-isomorphism between $\langle \lambda,\lhd_y\rangle$ and $\langle\gamma,<\rangle$ for some ordinal $\gamma\in[\lambda,\lambda^+)$ and this isomorphism witnesses that $x$ is an element of $\WW\OO_\lambda$ in $V$. 
   This shows that     $\WW\OO_\lambda^M\subseteq\WW\OO_\lambda$, $\mathsf{WO}(\vec{\nu},\vec{a})^M\subseteq\mathsf{WO}(\vec{\nu},\vec{a})$ and $\|b\|_{\vec{a}}=\|b\|_{\vec{a}}^M$ for all $b\in\mathsf{WO}(\vec{\nu},\vec{a})^M$. 
  Moreover, using \eqref{equation:EnumerateinLF} and the fact that $\lambda^+=(\lambda^+)^M$, we can pick a sequence $\seq{b_\gamma}{\lambda\leq\gamma<\lambda^+}$ with the property that for all $\gamma<\lambda^+$, the set $b_\gamma$ is the $<_{L[\VV]}$-least element of $\mathsf{WO}(\vec{\nu},\vec{a})^M$ with the property that $\|b_\gamma\|_{\vec{a}}=\gamma$ and 
  $b_\gamma(x(n))<\nu_{x(n+1)}$ for all $n<\omega$. 
  The following statement now follows from a combination of the above claims:

  \begin{claim*}
   The set $B=\Set{b_\gamma}{\lambda\leq\gamma<\lambda^+}$ is definable by a $\Sigma_1$-formula with parameters $\vec{\nu}$ and $x$. \qed 
  \end{claim*}

  Given $\lambda\leq\gamma<\lambda^+$, we let $c_\gamma$ denote the unique element of ${}^\omega\lambda$ such that the following statements hold for all $n<\omega$: 
  \begin{itemize}
   \item If $n$ is of the form $x(m+1)$ for some $m<\omega$, then $c_\gamma(n)=b_\gamma(x(m))$. 
   
   \item If $n\neq x(m+1)$ for all $m<\omega$, then $c_\gamma(n)=0$. 
  \end{itemize}
  We then know that $c_\gamma\in C(\vec{\nu})$ for all $\lambda\leq\gamma<\lambda^+$. 
  
   \begin{claim*}
   The set $C=\Set{c_\gamma}{\lambda\leq\gamma<\lambda^+}$ has cardinality $\lambda^+$ and is definable by a $\Sigma_1$-formula with parameters $\vec{\nu}$ and $x$. \qed 
  \end{claim*}

 Let $\map{\zeta}{B}{C}$ denote the unique function with $\zeta(b_\gamma)=c_\gamma$ for all $\lambda\leq\gamma<\lambda^+$.

  \begin{claim*}
   The map $\zeta$ is a homeomorphism of the subspace $B$ of $C(\vec{\nu})$ and the subspace $C$ of $C(\vec{\nu})$.  \qed 
  \end{claim*}

  Now, assume, towards a contradiction, that there is a perfect embedding $\map{\iota}{{}^\omega\lambda}{C(\vec{\nu})}$ with the property that $\ran{\iota}\subseteq C$. Then $\zeta^{{-}1}\circ\iota$ is a perfect embedding of ${}^\omega\lambda$ into $C(\vec{\nu})$ and $$\ran{\zeta^{{-}1}\circ\iota} ~ \subseteq ~ B ~ \subseteq ~ \mathsf{WO}(\vec{\nu},\vec{a})^M ~ \subseteq ~ \mathsf{WO}(\vec{\nu},\vec{a}).$$ 
  An application of Lemma \ref{lemma:BoundednessLemma+} now yields $c,d\in{}^\omega\lambda$ with $c\neq d$ and $$\|(\zeta^{{-}1}\circ\iota)(c)\|_{\vec{a}} ~ = ~ \|(\zeta^{{-}1}\circ\iota)(d)\|_{\vec{a}},$$ contradicting the definition of $B$. 
\end{proof}


\section{$\Sigma_1$-definability at I2-cardinals}\label{section:Positive}

Let $\map{j}{V}{M}$ be an I2-embedding with critical sequence $\vec{\lambda}=\seq{\lambda_n}{n<\omega}$ and set $\lambda=\sup_{n<\omega}\lambda_n$. Classical results (see \cite{MR0562614}) then show that $j$ is \emph{$\omega$-iterable}, {i.e.,} there exists a commuting system $$\langle\seq{M^j_\alpha}{\alpha\leq\omega},\seq{\map{j}{M^j_\alpha}{M^j_\beta}}{\alpha\leq\beta\leq\omega}\rangle$$ of inner models and elementary embeddings such that the following statements hold: 
 \begin{enumerate-(i)}
  \item $M^j_0=V$ and $j_{0,1}=j$. 
  
  \item If $n<\omega$, then $j_{n+1,n+2}=\bigcup\Set{j_{n,n+1}(j_{n,n+1}\restriction V_\alpha)}{\alpha\in\Ord}$. 
  
  \item $\langle M^j_\omega,\seq{j_{n,\omega}}{n<\omega}\rangle$ is a direct limit of $$\langle\seq{M^j_n}{n<\omega},\seq{\map{j_{m,n}}{M^j_m}{M^j_n}}{m\leq n<\omega}\rangle.$$ 
 \end{enumerate-(i)}
  Given $m\leq n<\omega$, we then have $V_\lambda\subseteq M^j_\omega\subseteq M^j_n\subseteq M^j_m$,  $\crit{j_{n,{n+1}}}=\lambda_n=j_{m,n}(\lambda_m)$, $j_{m,n}(\lambda)=\lambda$ and $j_{n,\omega}(\lambda_n)=\lambda$. 
 Moreover, it is easy to see that $j_{0,\omega}(\lambda^+)=\lambda^+$ holds and therefore $(2^\lambda)^{M^j_\omega}<\lambda^+$. 
 Note that the \emph{Mathias criterion} shows that the sequence $\vec{\lambda}$ is Prikry-generic over $M^j_\omega$ and, by the theory of Prikry-type forcings, this implies that $(2^\lambda)^{M^j_\omega[\vec{\lambda}]}<\lambda^+$.

  The following theorem is the main result of this section. We will later show that it is a direct strengthening of Theorem \ref{thm:Main1}.

  \begin{thm}\label{thm:Main1PlusParameters}
   Let $\map{j}{V}{M}$ be an I2-embedding with critical sequence $\vec{\lambda}=\seq{\lambda_n}{n<\omega}$ and let $N$ be an inner model of $\ZFC$ with $M^j_\omega\cup\{\vec{\lambda}\}\subseteq N$. Set $\lambda=\sup_{n<\omega}\lambda_n$. 
   If $X$ is a subset of $C(\vec{\lambda})$ with $\vert X\vert>(2^\lambda)^N$ that is definable over $V_\lambda$ by a $\Sigma^1_2$-formula with parameters in $V_{\lambda+1}^N$,  then there is a perfect embedding $\map{\iota}{{}^\omega\lambda}{C(\vec{\lambda})}$ with $\ran{\iota}\subseteq X$. 
  \end{thm}

 The proof of the above theorem closely follows the proof of Solovay's classical result showing that every $\mathbf{\Sigma}^1_2$-set of reals has the perfect set property if $\omega_1$ is inaccessible to the reals (see {\cite[Theorem 14.10]{MR1994835}}). The key ingredient that  makes this adaptation possible is an absoluteness theorem proven by Laver in \cite{La97}. 
 We start this argument by obtaining tree representations for the sets in the given definability class. 
 
  Given non-empty sets $a_0,\ldots,a_k$, a subset $T$ of ${}^{{<}\omega}a_0\times\ldots\times{}^{{<}\omega}a_k$ is a \emph{(descriptive) tree} if the following statements hold for all elements $\langle t_0,\ldots,t_k\rangle$ of $T$: 
  \begin{itemize}
   \item $\dom(t_0)=\ldots=\dom(t_k)$. 
   
   \item If $\ell\in\dom(t_k)$, then $\langle t_0\restriction\ell,\ldots,t_k\restriction\ell\rangle\in T$. 
  \end{itemize}
  In addition, if $T\subseteq{}^{{<}\omega}a_0\times\ldots\times{}^{{<}\omega}a_k$ is a tree, then we let $[T]$ denote the set of all tuples $\langle x_0,\ldots,x_k\rangle$ in ${}^\omega a_0\times\ldots\times{}^\omega a_k$ with the property that $\langle x_0\restriction\ell,\ldots,x_k\restriction\ell\rangle\in T$ holds for all $\ell<\omega$. 
  Finally, for every tree $T\subseteq{}^{{<}\omega}a_0\times\ldots\times{}^{{<}\omega}a_{k+1}$, we define $$p[T] ~ = ~ \Set{\langle x_0,\ldots,x_k\rangle\in{}^\omega a_0\times\ldots\times{}^\omega a_k}{\exists x_{k+1}\in{}^\omega a_{k+1} ~ \langle x_0,\ldots,x_{k+1}\rangle\in[T]}.$$
 As outlined in {\cite[Section 1]{La97}}, for singular strong limit cardinals $\lambda$ and $0<k<\omega$, there is a direct correspondence between subsets of $V_{\lambda+1}^k$ that are $\mathbf{\Sigma}^1_1$-definable over $V_\lambda$ and sets of the form $p[T]$ for trees $T\subseteq({}^{{<}\omega}V_\lambda)^{k+1}$. 
 Several key arguments in this section rely on the absoluteness properties of this correspondence that can be isolated from the arguments in {\cite[Section 1]{La97}}:

 \begin{lem}\label{lemma:AbsoSigma11}
  For every $\Sigma_1^1$-formula $\varphi(w_0,\ldots,w_{k+2})$ in the language of set theory with free second-order variables $w_0,\ldots,w_{k+2}$, there is a first-order formula $\psi(v_0,\ldots,v_{k+1})$ in  the language of set theory expanded by two unary relation symbols with free variables $v_0,\ldots,v_{k+1}$ such that $\ZFC$ proves that for every strictly increasing sequence $\vec{\lambda}=\seq{\lambda_n}{n<\omega}$ of strong limit cardinals  with supremum $\lambda$ and every $B\subseteq V_\lambda$, the set 
  \begin{equation}\label{equation:AbsoSigma11-1}
   T_{\varphi,B,\vec{\lambda}} ~ = ~ \Set{\langle t_0,\ldots,t_{k+1}\rangle\in({}^{{<}\omega}V_\lambda)^{k+2}}{\langle V_\lambda,\in,B,\vec{\lambda}\rangle\models\psi(t_0,\ldots,t_{k+1})}
  \end{equation}
   is a tree and the  following statements hold: 
  \begin{enumerate-(i)}
   \item\label{item:AbsoSigma11-2} $T_{\varphi,B,\vec{\lambda}}\cap({}^n V_\lambda)^{k+2}\in V_\lambda$ for all $n<\omega$. 
   
   \item\label{item:AbsoSigma11-3} If $\langle x_0,\ldots,x_k\rangle\in p[T_{\varphi,B,\vec{\lambda}}]$, $i\leq k$ and $m<n<\omega$, then $x_i(m)=x_i(n)\cap V_{\lambda_m}$. 
         
   \item\label{item:AbsoSigma11-4} The following statements are equivalent for all $A_0,\ldots,A_k\subseteq V_\lambda$: 
    \begin{enumerate}
      \item $\langle V_\lambda,\in\rangle\models\varphi(A_0,\ldots,A_k,B,\vec{\lambda})$. 
      
      \item There is $\langle x_0,\ldots,x_k\rangle\in p[T_{\varphi,B,\vec{\lambda}}]$ with $x_i(n)=A_i\cap V_{\lambda_n}$ for all $i\leq k$ and $n<\omega$. \qed
    \end{enumerate}
  \end{enumerate-(i)}
 \end{lem}

 The above lemma provides a setting in which a converse of Proposition \ref{prop:translate} holds:

 \begin{cor}\label{corollary:TranslateBack}
  For every $\Sigma^1_2$-formula $\psi(w_0,\ldots,w_{k-1})$ in the language of set theory with free second-order parameters $w_0,\ldots,w_{k-1}$, there exists a $\Sigma_1$-formula  $\varphi(v_0,\ldots,v_k)$ in the language of set theory such that $\ZFC$ proves that $$\varphi(A_0,\ldots,A_{k-1},V_\lambda,\vec{\lambda}) ~ \Longleftrightarrow ~ \langle V_\lambda,\in\rangle\models\psi(A_0,\ldots,A_{k-1})$$ holds for every strictly increasing sequence $\vec{\lambda}=\seq{\lambda_n}{n<\omega}$ of strong limit cardinals  with supremum $\lambda$ and  all $A_0,\ldots,A_{k-1}\in V_{\lambda+1}$.    \qed
 \end{cor}

 Following {\cite[Section 1]{La97}}, we now generalize the concept of \emph{Shoenfield trees} ({i.e.,} tree representations for $\mathbf{\Sigma}^1_2$-sets of real numbers) to higher cardinals of countable cofinality. 
  Given a tree $T\subseteq{}^{{<}\omega}a_0\times\ldots\times{}^{{<}\omega}a_k$, $i<k$ and $\langle s_0,\ldots,s_i\rangle\in {}^{{<}\omega}a_0\times\ldots\times{}^{{<}\omega}a_i$ with $\dom(s_0)=\ldots=\dom(s_i)$, we define $T^{\langle s_0,\ldots,s_i\rangle}$ to be the set of all tuples  $\langle t_{i+1},\ldots,t_k\rangle$ in ${}^{{<}\omega}a_{i+1}\times\ldots\times{}^{{<}\omega}a_k$ with the property that $$\dom(t_{i+1}) ~ = ~ \ldots ~ = ~ \dom(t_k) ~ \leq ~ \dom(s_0)$$ and $$\langle s_0\restriction\dom(t_k),\ldots, s_{i}\restriction\dom(t_k),t_{i+1},\ldots,t_k\rangle\in T.$$ 
  Note that $$T^{\langle s_0\restriction\ell,\ldots,s_i\restriction\ell\rangle} ~ = ~ T^{\langle s_0,\ldots,s_i\rangle}\cap({}^{{\leq}\ell}a_{i+1}\times\ldots\times{}^{{\leq}\ell}a_k)$$ holds for all $\ell\in\dom(s_0)$. 
  In addition, for every ordinal $\theta$, we let $R_{T,\theta}(s_0,\ldots,s_i)$ denote the set of functions $$\map{r}{T^{\langle s_0,\ldots,s_i\rangle}}{\theta}$$ satisfying $$r(\langle t_{i+1},\ldots,t_k\rangle) ~ < ~ r(\langle t_{i+1}\restriction\ell,\ldots,t_k\restriction\ell\rangle)$$ for all $\langle t_{i+1},\ldots,t_k\rangle\in T^{\langle s_0,\ldots,s_i\rangle}$ and $\ell<\dom(t_k)$. 
 It is then easy to see that  $r\restriction T^{\langle s_0\restriction\ell,\ldots,s_i\restriction\ell\rangle}$ is an element of $R_{T,\theta}(s_0\restriction\ell,\ldots,s_i\restriction\ell)$ for all $r\in R_{T,\theta}(s_0,\ldots,s_i)$ and $\ell<\dom(s_0)$.

 Now, let $\lambda$ be an infinite ordinal, let  $T\subseteq({}^{{<}\omega}V_\lambda)^{k+3}$ be a tree and let $\theta>\lambda$ be an ordinal. 
 We then define $S_{T,\theta}$ to be the subset of $({}^{{<}\omega}V_{\theta+\omega})^{k+2}$ consisting of all tuples $\langle s_0,\ldots,s_k,t\rangle$ such that the following statements hold: 
  \begin{itemize}
   \item $s_0,\ldots,s_k\in{}^{{<}\omega}V_\lambda$. 
   
   \item $\dom(s_0)=\ldots=\dom(s_k)=\dom(t)$. 
   
   \item There exists $s_{k+1}\in{}^{\dom(s_0)}V_\lambda$ and $r\in R_{T,\theta}(s_0,\ldots,s_{k+1})$ such that  
    \begin{equation}\label{equation:CodedLastCoordinate}
     t(\ell) ~ = ~ \langle s_{k+1}\restriction(\ell+1), ~ r\restriction T^{\langle s_0\restriction(\ell+1),\ldots,s_{k+1}\restriction(\ell+1)}\rangle
    \end{equation}
     holds for all $\ell\in\dom(t)$. 
  \end{itemize}
  It is then easy to check that  $S_{T,\theta}$ is a tree. 
 The following lemma from {\cite[Section 1]{La97}} shows how these constructions yield tree representations of $\mathbf{\Sigma}^1_2$-subsets of $V_{\lambda+1}$:

 \begin{lem}\label{lemma:Sigma12TreeRepresentations}
  Let $\varphi(w_0,\ldots,w_{k+3})$ be a $\Sigma_1^1$-formula  in the language of set theory with free second-order variables $w_0,\ldots,w_{k+3}$. 
  Then the following statements are equivalent for every strictly increasing sequence of inaccessible cardinals $\vec{\lambda}=\seq{\lambda_n}{n<\omega}$ with supremum $\lambda$, every limit ordinal $\theta\geq\lambda^+$ and all $A_0,\ldots,A_k,B\subseteq V_\lambda$:  
     \begin{enumerate-(i)}
      \item $\langle V_\lambda,\in\rangle\models\exists C ~ \neg\varphi(A_0,\ldots,A_k,B,C,\vec{\lambda})$. 
      
      \item There is $\langle x_0,\ldots,x_k\rangle\in p[S_{T_{\varphi,B,\vec{\lambda}},\theta}]$ with $x_i(n)=A_i\cap V_{\lambda_n}$ for all $i\leq k$ and $n<\omega$. \qed
    \end{enumerate-(i)}
 \end{lem}

Still following Laver's arguments, we now show that the structural properties of higher Shoenfield trees can be fruitfully combined with the combinatorics of I2-embeddings. The proof of the next lemma is a reformulation of the proof of {\cite[Theorem 1.4]{La97}}.

 \begin{lem}\label{lemma:I2Sigma12}
  Let $\varphi(w_0,\ldots,w_{k+3})$ be a $\Sigma_1^1$-formula  in the language of set theory with free second-order variables $w_0,\ldots,w_{k+3}$,  
  let $\map{j}{V}{M}$ be an I2-embedding with critical sequence $\vec{\lambda}=\seq{\lambda_n}{n<\omega}$ and let $N$ be an inner model of $\ZFC$ with $M^j_\omega\cup\{\vec{\lambda}\}\subseteq N$. Set $\lambda=\sup_{n<\omega}\lambda_n$ and $\theta=(\lambda^+)^V$. Fix $B\in V_{\lambda+1}^N$ and define $T=T_{\varphi,B,\vec{\lambda}}^V$. Then the following statements hold:
  \begin{enumerate-(i)}
   \item\label{item:AbsoShoen1} $T=T_{\varphi,B,\vec{\lambda}}^N$ and $S_{T,\theta}^N\subseteq S_{T,\theta}^V$. 
   
   \item\label{item:AbsoShoen2} There is an inclusion-preserving embedding $\map{\Lambda}{S_{T,\theta}^V}{S_{T,\theta}^N}$ with the property that for all $\langle s_0,\ldots,s_k,t\rangle\in S_{T,\theta}^V$, there exists $u$ with $$\Lambda(\langle s_0,\ldots,s_k,t\rangle) ~ = ~ \langle s_0,\ldots,s_k,u\rangle.$$ 
   
   \item\label{item:AbsoShoen3} $p[S_{T,\theta}^V]^V=p[S_{T,\theta}^N]^V$. 
   
   \item\label{item:AbsoShoen4} $p[S_{T,\theta}^V]^V\cap N=p[S_{T,\theta}^N]^N$. 
  \end{enumerate-(i)}
 \end{lem}
 
 \begin{proof}
  \ref{item:AbsoShoen1} Since $V_\lambda\cup\{B,\vec{\lambda}\}\subseteq N$, the fact that \eqref{equation:AbsoSigma11-1} holds in both $V$ and $N$ directly implies that $T=T_{\varphi,B,\vec{\lambda}}^N$. 
  In addition, if $s_0,\ldots,s_{k+1}\in {}^{{<}\omega}V_\lambda$ with $\dom(s_0)=\ldots=\dom(s_{k+1})$, then $$R_{T,\theta}(s_0,\ldots,s_{k+1})^N ~ \subseteq ~ R_{T,\theta}(s_0,\ldots,s_{k+1})^V.$$  In particular, we know that  $S_{T,\theta}^N\subseteq S_{T,\theta}^V$.

  \ref{item:AbsoShoen2} The proof of   {\cite[Theorem 1.4]{La97}} shows that for every  $d\in V_\lambda$ and every function $\map{f}{d}{\Ord}$, the function $\map{j_{0,\omega}\circ f}{d}{\Ord}$ is an element of $M^j_\omega$. 
  In particular, if  $s_0,\ldots,s_{k+1}\in {}^{{<}\omega}V_\lambda$ with $\dom(s_0)=\ldots=\dom(s_{k+1})$ and $r\in R_{T,\theta}(s_0,\ldots,s_{k+1})^V$, then the fact that $j_{0,\omega}(\theta)=\theta$ implies that $j_{0,\omega}\circ r\in R_{T,\theta}(s_0,\ldots,s_{k+1})^N$. 
  This inclusion allows us to define $\map{\Lambda}{S_{T,\theta}^V}{S_{T,\theta}^N}$ to be the unique function with the property that for all $\langle s_0,\ldots,s_k,t\rangle\in S_{T,\theta}^V$ and all $s_{k+1}\in{}^{\dom(s_0)}V_\lambda$ and $r\in R_{T,\theta}(s_0,\ldots,s_{k+1})$ such that  \eqref{equation:CodedLastCoordinate}  holds for all $\ell\in\dom(t)$, we have  $\Lambda(\langle s_0,\ldots,s_k,t\rangle) = \langle s_0,\ldots,s_k,u\rangle$, where $$u(\ell) ~ = ~ \langle s_{k+1}\restriction(\ell+1), ~ (j_{0,\omega}\circ r)\restriction T^{\langle s_0\restriction(\ell+1),\ldots,s_{k+1}\restriction(\ell+1)}\rangle$$ for all $\ell\in\dom(u)$. This definition directly ensures that $\Lambda$ is an inclusion-preserving embedding. 
  
    \ref{item:AbsoShoen3} Since  $S_{T,\theta}^N\subseteq S_{T,\theta}^V$, we know that  $p[S_{T,\theta}^N]^V\subseteq p[S_{T,\theta}^V]^V$. Pick a tuple  $\langle x_0,\ldots,x_k,y\rangle$ in $[S_{T,\theta}^V]^V$ and let $z$ be the unique element of ${}^\omega V_{\theta+\omega}$ with 
    \begin{equation}\label{equation:MoveBranch}
     \Lambda(\langle x_0\restriction n,\ldots,x_k\restriction n,y\restriction n\rangle) ~ = ~ \langle x_0\restriction n,\ldots,x_k\restriction n,z\restriction n\rangle
    \end{equation} 
    for all $n<\omega$. We then know that $\langle x_0,\ldots,x_k,z\rangle$ is an element of $[S_{T,\theta}^N]^V$. This shows that we also have $p[S_{T,\theta}^V]^V\subseteq p[S_{T,\theta}^N]^V$. 
    
      \ref{item:AbsoShoen4} First, the fact that $S_{T,\theta}^N\subseteq S_{T,\theta}^V$ directly implies that $p[S_{T,\theta}^N]^N \subseteq p[S_{T,\theta}^V]^V\cap N$. 
     Now, fix $\langle x_0,\ldots,x_k\rangle\in p[S_{T,\theta}^V]^V\cap N$ and pick $y\in V$ with $\langle x_0,\ldots,x_k,y\rangle\in [S_{T,\theta}^V]^V$.   
      Let $z$ denote the unique element of ${}^{{<}\omega}V_{\theta+\omega}$ such that \eqref{equation:MoveBranch} holds for all $n<\omega$.  
       This shows that $\langle x_0,\ldots,x_k,z\rangle\in [S_{T,\theta}^N]^V$. 
       Since the tuple $\langle x_0,\ldots,x_k\rangle$ is an element of $N$, we know that $$U ~ = ~ \Set{t\in{}^{{<}\omega}V_{\theta+\omega}}{\langle x_0\restriction\dom(t),\ldots,x_k\restriction\dom(t),t\rangle\in S_{T,\theta}^N}$$ is a tree of height $\omega$ in $N$ and $z\in[U]^V$. 
       In this situation, the fact that the ill-foundedness of $U$ is absolute between $N$ and $V$ yields an element $z^\prime$ of $[U]^N$. We then have  $\langle x_0,\ldots,x_k,z^\prime\rangle\in [S_{T,\theta}^N]^N$ and $\langle x_0,\ldots,x_k\rangle\in p[S_{T,\theta}^N]^N$.  
 \end{proof}

\begin{cor}[{\cite[Theorem 1.4]{La97}}]\label{cor:Sigma12correctness}
 Let $\varphi(w_0,\ldots,w_{n-1})$ be a $\Sigma_2^1$-formula  in the language of set theory with  free second-order variables $w_0,\ldots,w_{n-1}$. If  $\map{j}{V}{M}$ is an I2-embedding with critical sequence $\vec{\lambda}=\seq{\lambda_n}{n<\omega}$ and  $N$ is an inner model of $\ZFC$ with $M^j_\omega\cup\{\vec{\lambda}\}\subseteq N$, then the statement $$\langle V_\lambda,\in\rangle\models\varphi(A_0,\ldots,A_{n-1})$$ is absolute between $V$ and $N$ for all $A_0,\ldots,A_{n-1}\in V_{\lambda+1}^N$, where $\lambda=\sup_{n<\omega}\lambda_n$.  \qed 
\end{cor}
  
  In order to connect the above concepts with the existence of perfect subsets, we now adapt a classical result of Mansfield (see {\cite[Theorem 14.7]{MR1994835}}) to our setting:

\begin{lem}\label{lem:proj_Tree}
  Let $\vec{\lambda}=\seq{\lambda_n}{n<\omega}$ be  a strictly increasing sequence  of infinite cardinals with limit $\lambda$  and let $T\subseteq{}^{{<}\omega} a\times{}^{{<}\omega} b$ be a tree   with the property that $p[T]$ does not contain the range of a perfect embedding of ${}^\omega\lambda$ into ${}^\omega a$. 
 If $N$ is an inner model of $\ZFC$ with  $V_\lambda\cup\{T,\vec{\lambda}\}\subseteq N$, then $p[T]^V\subseteq N$.  
\end{lem}
	
	\begin{proof}
	 Given a tree $S\subseteq{}^{{<}\omega} a\times{}^{{<}\omega} b$, we define $S^\prime$ to be the set of all $\langle t,u\rangle\in S$ with the property that for all $n<\omega$, there exists $\dom(t)<\ell<\omega$ such that the set $$\Set{v\in{}^\ell a}{\exists w\in{}^\ell b ~ [t\subseteq v\wedge u\subseteq w\wedge\langle v,w\rangle\in S]}$$ has cardinality at least $\lambda_n$. 
	 Then it is easy to see that for every such tree $S$, the set $S^\prime$ is again a tree with $S^\prime\subseteq S$ and, if $S$ is an element of $N$, then $S^\prime$ is also contained $N$. 
	 Now, let $\seq{T_\alpha}{\alpha\in\Ord}$ denote the unique sequence of trees  with $T_0=T$, $T_{\alpha+1}=T_\alpha^\prime$ for all $\alpha\in\Ord$ and $T_\beta=\bigcap_{\alpha<\beta}T_\alpha$ for all $\beta\in\Lim$. 
	 Then it is easy to see that $T_\alpha\in N$ holds for all $\alpha\in\Ord$. Moreover, there exists $\alpha_*\in\Ord$ with $T_{\alpha_*}=T_\beta$ for all $\alpha_*\leq\beta\in\Ord$. Set $T_*=T_{\alpha_*}$. 
	 
	 \begin{claim*}
	  $T_*=\emptyset$. 
	 \end{claim*}
	 
	 \begin{proof}[Proof of the Claim]
	 Assume, towards a contradiction, that $T_*\neq\emptyset$. 
	 Let $S_{\vec{\lambda}}$ denote the subtree of ${}^{{<}\omega}\lambda$ consisting of all $s\in{}^{{<}\omega}\lambda$ with $s(\ell)<\lambda_\ell$ for all $\ell\in\dom(s)$. 
	 We inductively construct a system $\seq{\langle s_u,t_u\rangle\in T_*}{u\in S_{\vec{\lambda}}}$ such that the following statements hold for all $u,v\in S_{\vec{\lambda}}$: 
		\begin{itemize}
			\item If $u\subsetneq v$, then  $s_u\subsetneq s_v$ and $t_u\subsetneq t_v$. 
			
			\item If $\alpha<\beta<\lambda_{\dom(u)}$, then $\dom(s_{u^\frown\langle\alpha\rangle})=\dom(s_{u^\frown\langle\beta\rangle})$ and $s_{u^\frown\langle\alpha\rangle}\neq s_{u^\frown\langle\beta\rangle}$. 
		\end{itemize}
	 First, define $s_\emptyset=t_\emptyset=\emptyset$. 
	 Now, assume that $u\in S_{\vec{\lambda}}$ and $\langle s_u,t_u\rangle\in T_*$ is already constructed. 
	 Since  $\langle s_u,t_u\rangle\in T_*^\prime=T_*$, we can find $\dom(s_u)<\ell<\omega$ and a sequence $\seq{\langle s_\xi,t_\xi\rangle\in T_*}{\xi<\lambda_{\dom(u)}}$ with the property that for all $\xi<\rho<\lambda_{\dom(u)}$, we have $\dom(s_\xi)=\dom(s_\rho)=\ell$ and $s_\xi\neq s_\rho$. 
	 Given $\xi<\lambda_{\dom(u)}$, we then define  $s_{u^\frown\langle\xi\rangle}=s_\xi$ and $t_{u^\frown\langle\xi\rangle}=t_\xi$. 
	It then directly follows that the constructed sets satisfy all required  properties. This completes the inductive construction of our system. 
	If we now define $$\Map{\iota}{C(\vec{\lambda})}{{}^\omega a}{x}{\bigcup\Set{s_{x\restriction\ell}}{\ell<\omega}},$$ then our setup ensures  that $\iota$ is a perfect embedding.   Moreover, we have $$\langle\iota(x),\bigcup\Set{t_{x\restriction i}}{i<\omega}\rangle\in [T]$$ for all $x\in C(\vec{\lambda})$ and this shows that $\ran{\iota}$ is a subset of $p[T]$. Since there exists a perfect embedding of ${}^\omega\lambda$ into $C(\vec{\lambda})$, this yields a contradiction to our assumptions on $T$. 
	 \end{proof}

	 Now, fix $\langle x,y\rangle\in [T]$. Then there is an $\alpha<\alpha_*$ with $\langle x,y\rangle\in[T_\alpha]\setminus[T_{\alpha+1}]$ and we can find $k<\omega$ with the property that $\langle x\restriction k,y\restriction k\rangle\notin T_{\alpha+1}=T_\alpha^\prime$. 
	 Hence, there is   $n<\omega$ with the property  that for all $k<\ell<\omega$, the set $$E_\ell ~ = ~ \Set{s\in{}^\ell a}{\exists t\in{}^\ell b ~ [x\restriction k\subseteq s\wedge y\restriction k\subseteq t\wedge\langle s,t\rangle\in T_\alpha]}$$ has cardinality less than $\lambda_n$. 
	 Note that $x\restriction\ell\in E_\ell$ holds   for all $k<\ell<\omega$. 
	 Moreover, since $N$ contains the  sequence $\seq{E_\ell}{k<\ell<\omega}$ and each $E_\ell$ has cardinality less than $\lambda_n$ in $N$, we can find a sequence $\seq{\map{\tau_\ell}{\lambda_n}{E_\ell}}{k<\ell<\omega}$ of surjections that is an element of $N$. 
	 If we pick $z\in{}^\omega\lambda_n$ with $\tau_\ell(z(\ell))=x\restriction\ell$ for all $k<\ell<\omega$, then the fact that $V_\lambda\subseteq N$ ensures that $z$ is an element of $N$ and hence we can conclude that $x$ is also contained in the inner model $N$.  
\end{proof}

  We are now ready to prove the main result of this section:

  \begin{proof}[Proof of Theorem \ref{thm:Main1PlusParameters}]
     Let $\map{j}{V}{M}$ be an I2-embedding with critical sequence $\vec{\lambda}=\seq{\lambda_n}{n<\omega}$ and let $N$ be an inner model of $\ZFC$ with $M^j_\omega\cup\{\vec{\lambda}\}\subseteq N$. Set $\lambda=\sup_{n<\omega}\lambda_n$. 
   Fix a $\Sigma^1_1$-formula $\varphi(w_0,\ldots,w_3)$ with second-order variables $w_0,\ldots,w_3$ and $B\in V_{\lambda+1}^N$ such that the set $$X ~ = ~ \Set{A\in V_{\lambda+1}}{\langle V_\lambda,\in\rangle\models\exists C ~ \neg\varphi(A,B,C,\vec{\lambda})}$$ is a subset of $C(\vec{\lambda})$ of cardinality greater than $(2^\lambda)^N$.  
  Set  $T=T^V_{\varphi,B,\vec{\lambda}}$, $\theta=(\lambda^+)^V$, $S_1=S_{T,\theta}^V$ and $S_0=S_{T,\theta}^N\subseteq S_1$. 
   An application of Lemma \ref{lemma:I2Sigma12}.\ref{item:AbsoShoen3} then shows that $p[S_0]^V=p[S_1]^V$. In particular, since    Lemma \ref{lemma:Sigma12TreeRepresentations} ensures that every element of $X$ is of the form $\bigcup\Set{y(n)}{n<\omega}$ for some $y\in p[S_1]^V$, we know that $p[S_0]^V$ has cardinality greater than $(2^\lambda)^N$ in $V$ and we can conclude that $p[S_0]^V\nsubseteq N$. 
   In this situation, an application of Lemma \ref{lem:proj_Tree} shows that, in $V$, there exists a perfect embedding $\map{\iota}{{}^\omega\lambda}{{}^\omega V_\lambda}$ satisfying  $\ran{\iota}\subseteq p[S_0]=p[S_1]$. 
   
   Now, work in $V$ and define $Y$ to be the set of all $y\in{}^\omega V_\lambda$ with the property that $y(m)=y(n)\cap V_{\lambda_m}$ holds for all $m\leq n<\omega$ and $\bigcup\Set{y(n)}{n<\omega}$ is an element of $C(\vec{\lambda})$. Then $Y$ is a closed subset of ${}^\omega V_\lambda$ and $p[S_0]\subseteq Y$. Moreover, the map $$\Map{\tau}{Y}{C(\vec{\lambda})}{y}{\bigcup\Set{y(n)}{n<\omega}}$$ is a homeomorphism of the subspace $Y$ of ${}^\omega V_\omega$ and the space $C(\vec{\lambda})$ with $\tau[p[S_0]]=X$. In particular, there is a perfect embedding of ${}^\omega\lambda$ into $C(\vec{\lambda})$  whose range is contained  in $X$.  
  \end{proof}

  \begin{proof}[Proof of Theorem \ref{thm:Main1}]
  Let $\map{j}{V}{M}$ be an I2-embedding with critical sequence $\vec{\lambda}=\seq{\lambda_n}{n<\omega}$ and set $\lambda=\sup_{n<\omega}{\lambda_n}$.  Let $X$ be a subset of $\POT{\lambda}$ of cardinality greater than $\lambda$ that is definable by a $\Sigma_1$-formula with parameters in $V_\lambda\cup\{V_\lambda,\vec{\lambda}\}$. 
  In $M^j_\omega[\vec{\lambda}]$, there is an injective enumeration $\vec{e}=\seq{d_\alpha}{\alpha<\lambda}$ of $V_\lambda$ with the property that $V_{\lambda_n}=\Set{d_\alpha}{\alpha<\lambda_n}$ holds for all $n<\omega$. Define $Y$ to be the set of all $y\in C(\vec{\lambda})$ with the property that $y(0)=0$ and there exists $A\in X$ with $d_{y(n+1)}=A\cap V_{\lambda_n}$ for all $n<\omega$. Then $Y$ is a subset of $C(\vec{\lambda})$ of cardinality greater than $\lambda$ that is definable by a $\Sigma_1$-formula with parameters in $M^j_\omega[\vec{\lambda}]\cap V_{\lambda+1}$. 
  Since $(2^\lambda)^{M^j_\omega[\vec{\lambda}]}<\lambda^+$, we can now combine Theorem \ref{thm:Main1PlusParameters} with Proposition \ref{prop:translate} to find a perfect embedding of ${}^\omega\lambda$ into $C(\vec{\lambda})$ whose range is contained in $Y$. Using the fact that the subspace $X$ of $\POT{\lambda}$ is homeomorphic to the subspace $Y$ of $C(\vec{\lambda})$, we can now conclude that there is a perfect embedding of ${}^\omega\lambda$ into $\POT{\lambda}$ whose range is contained in $X$.  
  \end{proof}


\section{The $\vec{\mathcal{U}}$-Baire Property}

In \cite{Di23}, a new type of regularity property for higher function  spaces is introduced: the \emph{$\lambda$-Baire property}. 
We can formalize this regularity property in a natural way as the $\lambda$-generalization of the classical Baire category notions:

\begin{defn}[\cite{Di23}]
	Let \(X\) be a topological space and let $A$ be a subset of $X$.
	\begin{enumerate-(i)} 
         	 \item The set $A$ is \emph{\(\lb\)-meager in $X$} if it is a \(\lb\)-union of nowhere dense sets. 
	 
	         \item The set $A$ is  \emph{\(\lb\)-comeager in $X$}  if it is the	complement of a \(\lb\)-meager set, {i.e.,} if it contains the intersection of $\lambda$-many open dense subsets of $X$. 
	         
	     \item The space \(X\) is a \emph{\(\lb\)-Baire space} if every non-empty open subset of $X$ is not \(\lb\)-meager.    
	     
	     \item The set $A$ has the \emph{$\lambda$-Baire property in $X$} if there exists an open set $U$ in $X$ such that the symmetric difference $A \, \Delta \, U$ is $\lambda$-meager. 
	\end{enumerate-(i)}
\end{defn}

 Note that a space $X$ is a $\lambda$-Baire space if and only if the intersection of $\lambda$-many open dense sets is dense. 
 The definition of the $\vec{\mathcal{U}}$-Baire property is more complex, as a direct generalization is unfruitful\footnote{In \cite{Di23}, it is proven that the space \(C(\vec{\lb})\) is the $\omega_1$-union of nowhere dense sets}. It is strictly correlated to \emph{diagonal Prikry forcing} (see, for example, {\cite[Section 1.3]{MR2768695}}). In the following, fix a strictly increasing sequence $\vec{\lambda}=\seq{\lambda_n}{n<\omega}$  of measurable cardinals with limit $\lambda$ and a sequence $\vec{\mathcal{U}}=\seq{U_n}{n<\omega}$ with the property that $U_n$ is a normal ultrafilter on $\lambda_n$ for all $n<\omega$.

\begin{defn}
  The \emph{diagonal Prikry forcing with $\vec{\mathcal{U}}$} is the partial order  $\P$ defined by the following clauses: 
  \begin{enumerate-(i)}
   \item Conditions in $\P$ are sequences $p=\seq{p_n}{n<\omega}$ with the property that there exists $n<\omega$ such that $p_i<\lambda_i$ for all $i<n$ and $p_i\in U_i$ for all $n\leq i<\omega$. In this case, we set $s^p = \langle p_0,\dots,p_{n-1}\rangle$, $\lh(p) = \lh(s^p)$, and $A_i^p=p_i$ for all $n\leq i<\omega$. The sequence $s^p$ is also called the \emph{stem of $p$}. 
   
   \item Given conditions $p$ and $q$ in $\P$, we have $p\leq_{\P} q$ if and only if the following statements hold: 
    \begin{itemize}
    \item $\lh(p) \geq \lh(q)$. 
    
    \item $s^p$ is an end-extension of $s^q$.  
    
    \item If $\lh(q) \leq i < \lh(p)$, then   $s^p(i) \in A^q_i$. 
    
    \item If $\lh(p)\leq i<\omega$, then $A^p_i \subseteq A^q_i$. 
\end{itemize}
Moreover, we say that $p \leq^{*}_{\P} q$ if $p \leq_{\P} q$ and $\lh(p) = \lh(q)$.

  \end{enumerate-(i)}
\end{defn}

The intuition behind the definitions below  is the following: it is easy to see that the product topology on the classical Baire space is isomorphic to the topology of the maximal filters on Cohen forcing. Thus, we are going to define a topology on \(C(\vec{\lb})\) that is isomorphic to the topology of the maximal filters on the diagonal Prikry forcing. Note that we can define this only if $\lambda$ is limit of measurable cardinals, therefore this will be the only setting for which to consider our new regularity property. 

\begin{defn}\label{definition:U-BP}
 \begin{enumerate-(i)}
  \item Given a condition $p$ in $\P$, we define $$N_p ~ = ~  \Set{x \in \Clb}{\forall i<\omega ~ [i < \lh(p) \rightarrow x(i) = s^p(i) ~ \wedge ~  i\geq\lh(p) \rightarrow  x(i) \in
A^p_j]}.$$ 

  \item The \emph{Ellentuck-Prikry $\vec{\mathcal{U}}$-topology} on $\Clb$ (briefly, $\vec{\mathcal{U}}$-EP topology) is the topology whose basic open sets are of the form $N_p$ for some condition $p$ in $\P$. 
  
  \item A subset $A$ of $\Clb$ has the \emph{$\vec{\mathcal{U}}$-Baire property} if it has the   $\lb$-Baire property in the $\vec{\mathcal{U}}$-EP topology. 
 \end{enumerate-(i)}
\end{defn}

The results of \cite{Di23} now show that the constructed topological spaces possess properties that generalize key properties of the classical Baire space to $\lambda$:

\begin{prop}[$\lb$-Baire Category, \cite{Di23}]\label{proposition:LambdaBaireSpace}
	The space \(\Clb\) endowed with the $\vec{\mathcal{U}}$-EP topology is a \(\lb\)-Baire space. Moreover, every subset of \(\Clb\) that is \(\lb\)-comeager in the $\vec{\mathcal{U}}$-EP topology contains a basic open set of this topology.  
\end{prop}

To motivate the main results of this section, we first show that the above property is non-trivial:

\begin{thm}\label{theorem:NoBP}
 There exists a subset of $C(\vec{\lambda})$ without the $\vec{\mathcal{U}}$-Baire property. 
\end{thm}

 The fact that the $\vec{\mathcal{U}}$-EP topology is build using $2^\lambda$-many basic open subsets stops the proof of the above result from being a routine diagonalization argument. 
 Instead, we have to use strong combinatorial properties of $\P$ to reduce the class of relevant open subsets:

 \begin{lem}[Strong Prikry condition]
   If $D$ is a dense open subset of $\P$ and $p$ is a condition in $\P$, then there exists  a condition $q\leq^*_{\P}p$ and $n<\omega$  such that  $r\in D$ holds for every condition $r\leq_{\P}q$ with $\lh(r)\geq n$.   \qed  
 \end{lem}

 \begin{cor}\label{corollary:StrongPrikry}
  If $O$ is an open subset of $\P$ and $p$ is a condition in $\P$, then there exists a condition $\bar{p}\leq^*_{\P}p$ such that if there exists a condition $q \leq_{\P} \bar{p}$ with $q \in O$, then $r \in O$ holds for every $r \leq_{\P} \bar{p}$ with $\lh(r)\geq\lh(q)$.  \qed 
 \end{cor}

 Given a set $P$ of conditions in $\P$, we let $$U_P ~ = ~ \bigcup\Set{N_p}{p\in P} ~ \subseteq ~ C(\vec{\lambda})$$ denote the corresponding open set in the $\vec{\mathcal{U}}$-EP topology.

 \begin{prop}\label{proposition:PredenseDense}
  A set $P$ of condition in $\P$ is predense in the partial order $\P$ if and only if $U_P$ is  dense in the $\vec{\mathcal{U}}$-EP topology. 
 \end{prop}
 
 \begin{proof}
  First, assume that $P$ is predense in $\P$ and fix a condition $p$ in $\P$. Then there exists a condition $q$ in $P$ and a condition $r$ in $\P$ with $r\leq_\P p,q$. We now know that $\emptyset\neq N_r\subseteq N_p\cap N_q\subseteq N_p\cap U_P$. 
  
  Now, assume that $U_P$ is dense in the $\vec{\mathcal{U}}$-EP topology and fix a condition $p$ in $\P$. Since $N_p\cap U_P\neq\emptyset$, we can find a condition $q$ in $P$ and an element $x$ of $C(\vec{\lambda})$ with $x\in N_p\cap N_q$. Then there exists a condition $r$ in $\P$ with $r_i=x(i)$ for all $i<\max(\lh(p),\lh(q))$ and $A^r_i=A^p_i\cap A^q_i$ for all $\max(\lh(p),\lh(q))\leq i<\omega$. We then know that $r\leq_\P p,q$ holds. These computations show that $P$ is predense in $\P$. 
 \end{proof}

 \begin{lem}\label{lemma:SmallOpen}
  %
  If $U$ is an open set  in the $\vec{\mathcal{U}}$-EP topology, then there exists a set $P$ of at most $\lambda$-many conditions in $\P$  
  such that $U_P\subseteq U$ and $U\setminus U_P$ is nowhere dense in the $\vec{\mathcal{U}}$-EP topology.  
 \end{lem}
 
 \begin{proof}
  %
  %
  Define $O$ to be the set of all conditions $p$ in $\P$ with the property that $N_p\subseteq U$. 
  Then $O$ is an open subset of $\P$. In addition, define $S$ to be the set of all conditions $p$ in $\P$  such that $p_i=\lambda_i$ holds for all $\lh(p)\leq i<\omega$. 
  In this situation, Corollary \ref{corollary:StrongPrikry} shows that for every $p\in P$,  we can then find a condition $\bar{p}\leq_{\P}^*p$ with the property that if there is a condition $q \leq_{\P} \bar{p}$ with $q \in O$, then $r \in O$ holds for every $r \leq_{\P} \bar{p}$ with $\lh(r)\geq\lh(q)$. Define $$P ~ = ~ \Set{\bar{p}}{p\in S, ~ \exists q ~ [q\leq_{\P}\bar{p} ~ \wedge ~ q\in O]}.$$ The fact that the set $S$ has cardinality $\lambda$ then ensures that $P$ consists of at most $\lambda$-many conditions in $\P$.  
  %

 \begin{claim*}
  $U_P\subseteq U$.
 \end{claim*}
 
   \begin{proof}[Proof of the Claim]
    Pick $p\in S$ with the property that there is $q\leq_{\P}\bar{p}$ with $q\in O$ and fix $x\in N_{\bar{p}}$. 
    Then there exists a condition $r\leq_{\P}\bar{p}$ with $r_i=x(i)$ for all $i<\lh(q)$ and $r_i=A^{\bar{p}}_i$ for all $\lh(r)\leq i<\omega$. 
    Since $\lh(q)=\lh(r)$, we then know that $r\in O$ and $x\in N_r\subseteq U$.  
   \end{proof}

 \begin{claim*}
  If $p\in S$ with $q\notin O$ for all $q\leq_{\P}\bar{p}$, then $N_{\bar{p}}\cap U=\emptyset$. 
 \end{claim*}
 
   \begin{proof}[Proof of the Claim]
    Assume, towards a contradiction, that there is an $x\in N_{\bar{p}}\cap U$. Pick a condition $q$ in $\P$ with $x\in N_q\subseteq U$. Then there exists a condition $r$ in $\P$ with $r_i=x(i)$ for all $i<\max(\lh(\bar{p}),\lh(q))$ and $r_i=A^{\bar{p}}_i\cap A^q_i$ for all $\max(\lh(\bar{p}),\lh(q))\leq i<\omega$. We then know that $r\leq_\P \bar{p},q$ and $x\in N_r\subseteq N_q\subseteq U$. But this implies that  $r$ is an element of $O$ below $\bar{p}$, a contradiction.   
   \end{proof}

 Define $u$ to be the unique condition in $\P$ with $\lh(u)=0$ and $$A_i^u ~ = ~ \bigcap\Set{A_i^{\bar{p}}}{p\in S, ~ \lh(p)\leq i}$$ for all $i<\omega$. In addition, set $$N ~ = ~ \Set{x\in C(\vec{\lambda})}{\forall j<\omega ~ \exists j\leq i<\omega ~ x(i)\notin A_i^u}.$$

  \begin{claim*}
   The set $N$ is nowhere dense in the $\vec{\mathcal{U}}$-EP topology. 
  \end{claim*}
  
    \begin{proof}[Proof of the Claim]
      Assume, towards a contradiction, that $N$ is dense in $N_p$ for some  condition $p$ in $\P$. 
      Let $q\leq^*_{\P}p$  be the unique condition with $A^q_i=A^p_i\cap A^u_i$ for all $\lh(p)\leq i<\omega$. 
      Then there is $x\in N\cap N_q$ and we can find $\lh(q)\leq i<\omega$ with $x(i)\notin A^u_i$. But, this implies that $x(i)\notin A^q_i$, a contradiction.  
    \end{proof}

  \begin{claim*}
   $U\setminus U_P\subseteq N$. 
  \end{claim*}
  
  \begin{proof}[Proof of the Claim]
    Pick $x\in U\setminus U_P$ and fix $j<\omega$. Let $p$ denote the unique element of $S$ with $s^p=x\restriction j$. 
    Then $x\notin N_{\bar{p}}$, because otherwise we would have $x\in N_{\bar{p}}\cap U\neq\emptyset$ and our second claim would imply that $N_{\bar{p}}\subseteq U_P$. 
     Since $\bar{p}\leq^*_{\P}p$, we can now find $j\leq i<\omega$ with $x(i)\notin A^{\bar{p}}_i$. Our definitions then ensure that $A_i^u\subseteq A^{\bar{p}}_i$ and we can conclude that $x(i)\notin A^u_i$. These computations show that $x$ is an element of $N$. 
  \end{proof}
  
  This last claim completes the proof of the lemma. 
 \end{proof}

 \begin{proof}[Proof of Theorem \ref{theorem:NoBP}]
  Define $\mathcal{O}$ to be the collection of all subsets of $C(\vec{\lambda})$ of the form $U_P$ for some set $P$ of at most $\lambda$-many conditions in $\P$.   Then the set $\mathcal{O}$ has cardinality at most $2^\lambda$ and we can fix an enumeration $\seq{\langle U_\gamma,M_\gamma\rangle}{\gamma<2^\lambda}$  of all pairs $\langle U,M\rangle$ such that $U\in\mathcal{O}$ and there exists a sequence $\seq{O_\alpha}{\alpha<\lambda}$ of dense elements of $\mathcal{O}$  with $M=\bigcup_{\alpha<\lambda}(C(\vec{\lambda})\setminus O_\alpha)$. 
  We inductively define increasing sequences $\seq{A_\gamma}{\gamma<2^\lambda}$ and $\seq{B_\gamma}{\gamma<2^\lambda}$ of subsets of $C(\lambda)$ with $A_\gamma\cap B_\gamma=\emptyset$ and $\vert A_\gamma\cup B_\gamma\vert\leq\vert\gamma\vert$ for all $\gamma<2^\lambda$. Fix $\gamma<2^\lambda$ and assume that we already defined $A_\beta$ and $B_\beta$ for all $\beta<\gamma$. Set $A=\bigcup_{\beta<\gamma}A_\beta$ and $B=\bigcup_{\beta<\gamma}B_\beta$. Then $A\cap B=\emptyset$ and both sets have cardinality less than $2^\lambda$. 
  First, assume that $U_\gamma$ is empty. Since Proposition \ref{proposition:LambdaBaireSpace} ensures that $C(\vec{\lambda})\setminus M_\gamma$ has cardinality $2^\lambda$, we can find $x\in C(\vec{\lambda})\setminus (B\cup M_\gamma)$. We then define $A_\gamma=A\cup\{x\}$ and $B_\gamma=B$. 
  Next, assume that $U_\gamma$ is non-empty. Then  Proposition \ref{proposition:LambdaBaireSpace} shows that $U_\gamma\setminus M_\gamma$ has cardinality $2^\lambda$ and we can find $x\in U_\gamma\setminus (A\cup M_\gamma)$. We now define $A_\gamma=A$ and  $B_\gamma=B\cup\{x\}$. This completes our construction.

  Define  $A=\bigcup_{\gamma<2^\lambda}A_\gamma$ and $B=\bigcup_{\gamma<2^\lambda}B_\gamma$. Then $A\cap B=\emptyset$. Assume, towards a contradiction, that the set $A$ has the $\vec{\mathcal{U}}$-Baire property. Pick an open subset $U$ in the $\vec{\mathcal{U}}$-EP topology such that $A \, \Delta \, U$ is $\lambda$-meager in this topology. 
  Then Lemma \ref{lemma:SmallOpen} shows that there exists $W\in\mathcal{O}$ with $W\subseteq U$ and $U\setminus W$ nowhere dense. It follows that $A \, \Delta \, W$ is also $\lambda$-meager. 
  Another application of Lemma \ref{lemma:SmallOpen} then yields a  sequence $\seq{O_\alpha}{\alpha<\lambda}$ of dense elements of $\mathcal{O}$ with $A \, \Delta \, W\subseteq\bigcup_{\alpha<\lambda}(C(\vec{\lambda})\setminus O_\alpha)$. In this situation, there exists a    $\gamma<2^\lambda$ with $U_\gamma=W$ and $M_\gamma=\bigcup_{\alpha<\lambda}(C(\vec{\lambda})\setminus O_\alpha)$. 
  Then $U_\gamma\neq\emptyset$, because otherwise our construction would ensure that there is $x\in A\setminus W$ with $x\notin M_\gamma$. 
  But this means that there is $x\in B\cap U_\gamma$ with $x\notin M_\gamma$ and therefore $x\in A\cap B$, a contradiction.  
 \end{proof}

We now proceed by showing that, in the model constructed in the proof of Theorem \ref{theorem:NegativeI2Clambda}, the above constructions can also be used to find a simply definable set without the  $\vec{\mathcal{U}}$-Baire property:

\begin{thm}\label{theorem:Sigma1No-UBP}
  If $\map{j}{V}{M}$ is an I2-embedding whose critical sequence has supremum $\lambda$, then the following statements hold in an inner model:   
 \begin{enumerate-(i)}
  \item There is an I2-embedding whose critical sequence has supremum $\lambda$. 
  
  \item If $\vec{\lambda}=\seq{\lambda_n}{n<\omega}$ is a strictly increasing sequence of measurable cardinal with limit $\lambda$ and $\vec{\mathcal{U}}=\seq{U_n}{n<\omega}$ is a sequence  with the property that $U_n$ is a normal ultrafilter on $\lambda_n$ for all $n<\omega$, then there is a subset $z$ of $\lambda$ and a subset $X$ of $C(\vec{\lambda})$ such that $X$ does not have the $\vec{\mathcal{U}}$-Baire property  and the set $X$ is definable by a $\Sigma_1$-formula with parameter $z$. 
  \end{enumerate-(i)}
\end{thm}

\begin{proof}
  As in the proof of Theorem \ref{theorem:NegativeI2Clambda}, pick a subset $y$ of $\lambda$ with $V_\lambda\cup\{j\restriction V_\lambda\}\subseteq L[y]$ and work in $L[y]$. Then there is an I2-embedding whose critical sequence has supremum $\lambda$. 
 Fix a strictly increasing sequence $\vec{\lambda}=\seq{\lambda_n}{n<\omega}$ of measurable cardinals with limit $\lambda$ and a sequence $\vec{\mathcal{U}}=\seq{U_n}{n<\omega}$ with the property that $U_n$ is a normal ultrafilter on $\lambda_n$ for all $n<\omega$. 
 We can now find an unbounded subset $z$ of $\lambda$ with the property that the $\{\vec{\mathcal{U}}\}$ is definable by a $\Sigma_1$-formula with parameter $z$ and there is a well-ordering $\lhd$ of $\HHrm{\lambda^+}$ of order-type $\lambda^+$ with the property that the set of all proper initial segments of $\lhd$ is definable by a $\Sigma_1$-formula with parameter $z$. 
 It then directly follows that the set $\{\vec{\mathcal{U}}\}$, the set of all conditions in $\P$, the ordering of $\P$, the compatibility relation of $\P$ and the incompatibility relation of $\P$ are all $\Delta_1$-definable from the parameter $z$.

 Now, define $\mathcal{O}$ to be the set of all pairs $\langle P,\vec{Q}\rangle$ with the property that $P$ is a set of at most $\lambda$-many conditions in $\P$ and $\vec{Q}=\seq{Q_\alpha}{\alpha<\lambda}$ is a sequence with the property that each $Q_\alpha$ is a set of at most $\lambda$-many conditions in $\P$. It is then easy to see that $\mathcal{O}$ is a subset of $\HHrm{\lambda^+}$ of cardinality  $\lambda^+$ that is definable by a $\Sigma_1$-formula with parameter $z$. 
 Let $\seq{\langle P_\gamma,\seq{Q^\gamma_\alpha}{\alpha<\lambda}\rangle}{\gamma<\lambda^+}$ denote the enumeration of $\mathcal{O}$ induced by $\lhd$. We then again know that this sequence is definable by a $\Sigma_1$-formula with parameter $z$. 
  Arguing as in the proof of Theorem \ref{theorem:NoBP}, we can now use Proposition \ref{proposition:PredenseDense} to show that for every $\gamma<\lambda^+$ with the property that  $Q^\gamma_\alpha$ is predense in $\P$ for all $\alpha<\lambda$, the set $\bigcap_{\alpha<\lambda}U_{Q^\gamma_\alpha}$ has cardinality $\lambda^+$.  
  Moreover, we know that for every $\gamma<\lambda^+$ with the property that $P_\gamma\neq\emptyset$ and $Q^\gamma_\alpha$ is predense in $\P$ for all $\alpha<\lambda$, the set $U_P\cap\bigcap_{\alpha<\lambda}U_{Q^\gamma_\alpha}$ has cardinality $\lambda^+$.  This shows that there is a unique sequence $\seq{d_\gamma}{\gamma<\lambda^+}$ with the property that for all $\gamma<\lambda^+$, the set $d_\gamma$ is the $\lhd$-least element of $\HHrm{\lambda^+}$ such that one of the following statements hold: 
  \begin{itemize}
   \item The set $d_\gamma$ is of the form $\langle 0,p\rangle$, where $p$ is a condition in $\P$ with the property that there exists an $\alpha<\lambda$ such that all conditions in $Q^\gamma_\alpha$ are incompatible with $p$ in $\P$. 
   
   \item $P_\gamma$ is the empty set and the set $d_\gamma$ is of the form $\langle 1,x\rangle$, where $x$ is an element of $\bigcap_{\alpha<\lambda}U_{Q^\gamma_\alpha}$ with the property that $x\neq v$ holds whenever $\beta<\gamma$ and $d_\beta$ is of the form $\langle 2,v\rangle$ for some $v$ in $C(\vec{\lambda})$. 
   
   \item The set $d_\gamma$ is of the form $\langle 2,x\rangle$, where $x$ is an element of $U_{P_\gamma}\cap\bigcap_{\alpha<\lambda}U_{Q^\gamma_\alpha}$ with the property that $x\neq u$ holds whenever $\beta<\gamma$ and $d_\beta$ is of the form $\langle 1,u\rangle$ for some $u$ in $C(\vec{\lambda})$. 
  \end{itemize}
  
  This definition then ensures that the sequence $\seq{d_\gamma}{\gamma<\lambda^+}$ is definable by a $\Sigma_1$-formula with parameter $z$. We define $$A ~ = ~ \Set{x\in C(\vec{\lambda})}{\exists\gamma<\lambda^+ ~ d_\gamma=\langle 1,x\rangle}.$$ Then $A$ is definable by a $\Sigma_1$-formula with parameter $z$ and, by repeating the computations made in the proof of Theorem \ref{theorem:NoBP}, we can show that $A$ does not have the $\vec{\mathcal{U}}$-Baire property. 
\end{proof}

Contrary to the perfect set property case, there are no previous results about the possibility of $\Sigma_1$- or $\Sigma^1_2$-definable sets to have this kind of regularity  property. In the following, we will again  focus on the structural consequences of large cardinal assumptions close to the Kunen inconsistency. 
The following lemma will allow us to prove  an analogue to Theorem \ref{thm:Main1PlusParameters} for the $\vec{\mathcal{U}}$-Baire property:

\begin{lem}\label{lemma:GenericsComeager}
 Let $\vec{\lambda}=\seq{\lambda_n}{n<\omega}$ be a strictly increasing sequence of measurable cardinals with supremum $\lambda$ and let $N$ be an inner model of $\ZFC$ with $V_\lambda\cup\{\vec{\lambda}\}\subseteq N$ and $(2^\lambda)^N<\lambda^+$. 
 If $\vec{\mathcal{U}}=\seq{U_n}{n<\omega}$ is a sequence in $N$ with the property that  $U_n$ is a normal ultrafilter on $\lambda_n$ for all $n<\omega$ and $$C ~ = ~  \Set{x \in C(\lb)}{\anf{\textit{$x$ is $\mathbb{P}_{\vec{\mathcal{U}}}^N$-generic over $N$}}},$$ then $C$ is \(\lb\)-comeager in  the $\vec{\mathcal{U}}$-EP topology. 
\end{lem}

\begin{proof}
     By the \emph{Mathias condition} for the diagonal Prikry forcing (see {\cite{MR2193185}}), the set $C$ consists of all $x\in C(\vec{\lambda})$ with the property that for every sequence $\vec{A}= \seq{A_n \in U_n}{n < \omega}$ in $N$, the function $x$   belongs to the  dense open set $$\Set{x \in \Clb}{\exists m < \omega ~  \forall m\leq n<\omega ~ x(n) \in A_n}.$$ 
 Since $(2^\lambda)^N<\lambda^+$, there are only $\lambda$-many dense open sets of this form and Proposition \ref{proposition:LambdaBaireSpace} yields the desired conclusion.  
\end{proof}

We are now ready to prove our analogue to Theorem \ref{thm:Main1PlusParameters}:

\begin{thm}\label{thm:BP for lightface}
Let $\map{j}{V}{M}$ be an I2-elementary embedding  with \(\lb\) being the supremum of its critical sequence $\vec{\lb}=\seq{\lb_n}{n < \o}$ and let $N$ be an inner model of $\ZFC$ with $M^j_\omega\cup\{\vec{\lambda}\}\subseteq N$ and $(2^\lambda)^N<\lambda^+$. 
 Then there exists a sequence \(\vec{\mathcal{F}}=\seq{F_n}{n<\omega}$ in $N$ such that each \(F_n\) is a normal ultrafilter on \(\lb_n\) and every subset of \(\Clb\) that is definable over $V_\lambda$ by a $\Sigma^1_2$-formula with parameters in $V^N_{\lambda+1}$   has the $\vec{\mathcal{F}}$-Baire property.
\end{thm}
	
\begin{proof}
 Since $V_\lambda\subseteq M^j_\omega\subseteq N$, $\vec{\lambda}\in N$ and each $\lambda_n$ is a measurable cardinal in $N$, we can pick a sequence \(\vec{\mathcal{F}}=\seq{F_n}{n<\omega}$ in $N$ such that each \(F_n\) is a normal ultrafilter on \(\lb_n\). Note that every condition in $\mathbb{P}_{\vec{\mathcal{F}}}^N$ is a condition in $\mathbb{P}_{\vec{\mathcal{F}}}^V$. In $V$, we define $$C ~ = ~  \Set{x \in C(\lb)}{\anf{\textit{$x$ is $\mathbb{P}_{\vec{\mathcal{F}}}^N$-generic over $N$}}}.$$ Then Lemma \ref{lemma:GenericsComeager} shows that   \(C\) is \(\lb\)-comeager in  the $\vec{\mathcal{F}}$-EP topology.

    Fix a $\Sigma^1_2$-formula $\varphi(w_0,w_1)$ with second-order variables $w_0$ and $w_1$ and $B\in V_{\lambda+1}^N$ such that the set $$X ~ = ~ \Set{A\in V_{\lambda+1}}{\langle V_\lambda,\in\rangle\models\varphi(A,B)}$$ is a subset of $C(\vec{\lambda})$.  
    Define $O$ to be the set of all conditions $p$ in $\mathbb{P}_{\vec{\mathcal{F}}}^N$ with 
    \begin{equation}\label{equation:ForcingStatement}
     p \Vdash^N_{\mathbb{P}_{\vec{\mathcal{F}}}^N}\anf{\langle V_{\check{\lambda}},\in\rangle\models\varphi(\dot{x},\check{B})},
    \end{equation} 
      where $\dot{x}$ denotes the canonical $\mathbb{P}_{\vec{\mathcal{F}}}^N$-name  for the generic sequence in $N$.

 Work in $V$ and define $U$ to be the union of all sets of the form $N_p$ with $p\in O$. Fix $x\in C$. 
 First, assume that $x\in U$ and fix $p\in O$ with $x\in N_p$. Since $$G_x ~ = ~ \Set{p\in \mathbb{P}_{\vec{\mathcal{F}}}^N}{x\in N_p}$$ is the filter on $\mathbb{P}_{\vec{\mathcal{F}}}^N$ induced by $x$, we then know that  
  \begin{equation}\label{equation:DefiningFormula}
  \langle V_\lambda,\in\rangle\models\varphi(x,B)
 \end{equation}
  holds in $N[x]$ and therefore Corollary \ref{cor:Sigma12correctness} shows that $x$ is an element of $X$. In the other direction, assume that  $x\in X$. Then \eqref{equation:DefiningFormula} holds in $V$ and Corollary \ref{cor:Sigma12correctness} ensures that this statement also holds in $N[x]$. Then there is a condition $p$ in $G_x$ with the property that \eqref{equation:ForcingStatement} holds. But then $p\in O$, $x\in N_p$ and hence $x\in U$. 
  These computations now show that the sets $U$ and $C(\vec{\lambda})\setminus C$ witness that $X$ has the $\vec{\mathcal{F}}$-Baire property.  
\end{proof}

 A quick analysis of the proof shows that the consequences of the above theorem hold for every $\vec{\mathcal{F}}\in N$.

 Note that, since $(2^\lambda)^{M^j_\omega[\vec{\lambda}]}<\lambda^+$ holds in the situation of the above theorem, there exists a sequence $\vec{\mathcal{F}}$ of normal measures such that every subset of \(\Clb\) that is definable over $V_\lambda$ by a $\Sigma^1_2$-formula with parameters in $V_\lambda\cup\{\vec{\lambda}\}$   has the $\vec{\mathcal{F}}$-Baire property. 

 In the remainder of this paper, we study the interaction  of I0-embeddings with the $\lambda$-Baire property of families of sets. 
  One of the key ingredients of the proof of Theorem \ref{thm:BP for lightface} is Corollary \ref{cor:Sigma12correctness}, that states that there is a certain amount of absoluteness between $V$ and models that contain \(M_\o^j[\vec{\lambda}]\).   Woodin and Cramer  proved that  I0-embeddings also entail absoluteness-like results.
  
  Remember that, given a limit ordinal $\lambda$, we define $$\Theta^{L(V_{\lambda+1})} ~ = ~ \sup\Set{\alpha\in\Ord}{\textit{There is a surjection $\map{\pi}{V_{\lambda+1}}{\alpha}$ in $L(V_{\lambda+1})$}}.$$ 
  This concept generalizes the definition of $\Theta$ for $L(\mathbb{R})$. Since $L(\mathbb{R})$  is not going to appear in this paper and there is no risk of confusion, we will below write $\Theta$ instead of $\Theta^{L(V_{\lambda+1})}$. 
  An ordinal $\alpha<\Theta$ is  called \emph{good} if every element of $L_\alpha(V_{\lb+1})$ is definable over $L_\alpha(V_{\lb+1})$ from an element of $V_{\lb+1}$. 
 The next theorem is called  \emph{Generic Absoluteness} in \cite{Woo11}:

\begin{thm}[Woodin, {\cite[Theorem 82]{Cra17}}]\label{thm:Cramer} 
  Let $\map{j}{L(V_{\lambda+1})}{L(V_{\lambda+1})}$ be an I0-embedding that is $\omega$-iterable and let 
   $\map{j_{0,\o}}{L(V_{\lb+1})}{M_\o}$ be the embedding into the $\o$-th iterate of $L(V_{\lb+1})$ by $j$. 
   Assume that $\mathbb{P}\in M_\omega$ is a partial order and $g\in V$ is $\mathbb{P}$-generic over $M_\omega$ with $\cof{\lambda}^{M_\omega[g]}=\omega$. 
   If $\alpha<\Theta$ is good, then for some $\bar \alpha < \lb$, there is an elementary embedding $$\map{\pi}{L_{\bar \alpha}(M_\omega[g]\cap V_{\lb+1})}{L_\alpha(V_{\lb+1})}$$ that is the identity below $\lb$. 
\end{thm}

Note that, in the situation of the above theorem the good ordinals are cofinal in $\Theta$ (see \cite{La01}). 
Moreover, if there exists an I0-embedding, then there exists an iterable I0-embedding (see \cite[Lemma 10, Lemma 21]{Woo11}).  Therefore, the hypothesis of the above result is not restrictive.

\begin{thm}\label{thm:I0-lightface-BP}
 Let $\map{j}{L(V_{\lambda+1})}{L(V_{\lambda+1})}$ be an I0-embedding with critical sequence $\vec{\lambda}=\seq{\lambda_n}{n<\omega}$. 
  Then there exists a sequence $\vec{\mathcal{F}}=\seq{F_n}{n<\omega}$ such that  each $F_n$ is a normal ultrafilter on $\lambda_n$ and every  subset of $\Clb$ that is definable over $V_\lambda$ by a $\Sigma^1_n$-formula with parameters in $V_{\lb+1}$  has the $\vec{\mathcal{F}}$-Baire property.  
\end{thm}

\begin{proof} 
 By earlier remarks, we may assume that $j$ is $\omega$-iterable. In the following, we let $\map{j_{0,\o}}{L(V_{\lb+1})}{M_\o}$ denote the embedding into the $\o$-th iterate of $L(V_{\lb+1})$ by $j$. 
  Then $\vec{\lambda}$ is Prikry-generic over $M_\omega$ and there is a sequence $\vec{\mathcal{F}}=\seq{F_n}{n<\omega}$ in $M_\omega[\vec{\lambda}]$ such that  each $F_n$ is a normal ultrafilter on $\lambda_n$. Finally, we define $\mathbb{P}$ to be the corresponding diagonal Prikry forcing $\mathbb{P}_{\vec{\mathcal{F}}}^{M_\omega[\vec{\lambda}]}$ in $M_\omega[\vec{\lambda}]$.

Given $n<\omega$, we fix a $\Sigma^1_n$-formula $\varphi(w_0,w_1)$ in the language of set theory with free second-order variables $w_0$ and $w_1$. 
Given $y \in M_\o[\vec{\lb}] \cap V_{\lb+1}$, we define $$X_{\varphi,y} ~ = ~ \Set{x \in \Clb}{\langle V_{\lb},\in\rangle \models \varphi(x,y)}.$$
As in the proof of Theorem \ref{thm:BP for lightface}, we now define $O_{\varphi,y}$ to be the open subset of $\mathbb{P}$ in $M_\omega[\vec{\lambda}]$ that consists of all conditions $p$ with $$p \Vdash^{M_\o[\vec{\lb}]}_{\mathbb{P}} \anf{\langle V_{\check{\lambda}},\in\rangle\models\varphi(\dot{x},\check{y})},$$
where $\dot{x}$ denotes the canonical $\mathbb{P}$-name for the generic sequence in $M_\omega[\vec{j}]$. 
 In addition, we let $U_{\varphi,y}$ denote the union of all sets $N_p$ with $p\in O_{\varphi,y}$ in $V$. 
 Finally, we define $C$ to be the set of all $x$ in $\Clb$ that are $\mathbb{P}$-generic over $M_\omega[\vec{\lambda}]$. Since $(2^\lambda)^{M_\omega[\vec{\lambda}]}<\lambda^+$, an application of Lemma \ref{lemma:GenericsComeager} shows that  $C$ is \(\lb\)-comeager in the $\vec{\mathcal{F}}$-EP topology.

 Now, fix $x$ in $C$.  Still following the proof of Theorem \ref{thm:BP for lightface}, we then know that $x$ is an element of $U_{\varphi,y}$ if and only if $$\langle V_\lambda,\in\rangle\models\varphi(x,y)$$ holds in $M_\omega[\vec{\lambda},x]$. 
  The model \(M_\o[\vec{\lb},x]\) is a generic extension (via the forcing that is a two-step iteration of Prikry and diagonal Prikry forcing) of $M_\omega$ and $\cof{\lambda}^{M_\omega[\vec{\lambda},x]}=\omega$. 
   Therefore, we can apply Generic Absoluteness to \(M_\o[\vec{\lb},x]\) to show that $x\in U_{\varphi,y}$ if and only if  \(x \in X_{\varphi,y} \). 
   These computations show that $U_{\varphi,y} \, \Delta \, X_{\varphi,y}\subseteq C(\vec{\lambda})\setminus C$ and we can conclude that the set $X_{\varphi,y}$  has the $\vec{\mathcal{F}}$-Baire property.

We now know that the statement  
 \begin{equation}\label{equation:BP}
  \anf{\textit{$X_{\varphi,y}$ has the $\vec{\mathcal{F}}$-Baire property}}
  \end{equation}
   holds in $L(V_{\lb+1})$ for every $\Sigma^1_n$-formula $\varphi(w_0,w_1)$ and for all $y \in M_\o[\vec{\lb}] \cap V_{\lb+1}$. 
 We claim that this statement can be expressed by a formula that only uses a single existential quantifier bounded by the set $V_{\lambda+2}$ of all subsets of $V_{\lb+1}$. 
  Notice that, as a consequence of this, it follows that the $\vec{\mathcal{F}}$-Baire property is upward absolute. 
 By definition of $\vec{\mathcal{F}}$-Baire property, the set $X_{\varphi,y}$ has the $\vec{\mathcal{F}}$-Baire property if and only if there exist an open subset $U$ of $\Clb$ and a sequence $\seq{C_\alpha}{\alpha < \lb}$ of closed nowhere dense subsets of $\Clb$ with the property that $A \, \Delta \, U \subseteq \bigcup_{\alpha < \lb} C_\alpha$.  
 Notice now that each open set $W$ is determined by the subset $\Set{p\in\mathbb{P}_{\vec{\mathcal{F}}}}{N_p \subseteq W}$ of $V_{\lb+1}$. 
 %
Hence, the set $U$ and the sequence $\seq{C_\alpha}{\alpha < \lb}$ can be determined by a $\lb$-sequence of subsets of $V_{\lb+1}$, which in turn can be canonically identified with a subset of $V_{\lb+1}$. It is now easy to see that the claim holds.

Now, given  $y \in V_{\lb+1}$ with the property that \eqref{equation:BP} holds in $L(V_{\lambda+1})$, we define $\alpha_y$ to be the least ordinal $\alpha$ below $\Theta$ such that \eqref{equation:BP} holds in $L_\alpha(V_{\lambda+1})$. Such an ordinal exists below $\Theta$ because all the subsets of $V_{\lambda+1}$ in $L(V_{\lambda+1})$ are elements of $L_\Theta(V_{\lambda+1})$ (see, for example, {\cite[Lemma 5.6]{Di18}}). In addition, we define $\alpha_y=0$ for all $y \in V_{\lb+1}$ with the property that \eqref{equation:BP} fails in $L(V_{\lambda+1})$. 
 The resulting function $y\mapsto\alpha_y$ is then definable  in $L(V_{\lb+1})$. We now want to prove that $$\alpha ~ = ~ \sup\Set{\alpha_y}{y \in V_{\lb+1}} ~ < ~ \Theta.$$ 

For any $x \in V_{\lb+1}$, we define $<_x$ to be  the canonical well-ordering of $\HOD^{L(V_{\lambda+1})}_{x}$, the inner model of all sets hereditarily definable in $L(V_{\lambda+1})$ with ordinals and $x$ as parameters.\footnote{It is a standard argument that $L(V_{\lambda+1})=\bigcup\Set{\HOD^{L(V_{\lambda+1})}_{x}}{x\in V_{\lambda+1}}$} 
In addition, for all $x,y\in V_{\lb+1}$, we let $g_x(y)$ denote the $<_x$-smallest surjection from $V_{\lb+1}$ to $\alpha_y$, if  it  exists and  otherwise $g_x(y)=0$. 
The  map $x \mapsto g_x$ is also definable  in $L(V_{\lb+1})$. 
It is now easy to see that the function $f$ defined by  
$$
f(x,y,z)=
\begin{cases}
    g_x(y)(z), & \text{if } g_x(y)\neq 0\\
    0, & \text{otherwise}
\end{cases}
$$
is a surjection from $V_{\lb+1}^3$ to $\alpha$ and hence $\alpha<\Theta$. 

%
In particular, we know that \eqref{equation:BP} holds in $L_\alpha(V_{\lambda+1})$ for all $y \in M_\o[\vec{\lb}] \cap V_{\lb+1}$. 
By the fact that the sequence of good ordinals is cofinal in $\Theta$ and that the $\vec{\mathcal{F}}$-Baire property is upward absolute, we can assume that $\alpha$ is good. 
Then, by Theorem \ref{thm:Cramer}, there exist $\bar \alpha<\lb$ and an elementary embedding $$\map{\pi}{L_{\bar \alpha}(M_\o[\vec{\lb}] \cap V_{\lb+1})}{L_{\alpha}(V_{\lb+1})}$$ such that $\pi \rest (M_\o[\vec{\lb}] \cap V_{\lb+1})=\text{id}_{M_\o[\vec{\lb}] \cap V_{\lb+1}}$. 
Thus,   we can conclude that 
\begin{equation*}
 \begin{split}
   & \forall y\in  M_\o[\vec{\lb}] \cap V_{\lb+1} ~   L_{\alpha}(V_{\lb+1}) \models \anf{\textit{$X_{\varphi,y}$ has the $\vec{\mathcal{F}}$-Baire property}} \\
   \Longleftrightarrow ~ {} ~ &  \forall y\in  M_\o[\vec{\lb}] \cap V_{\lb+1}  ~ L_{\bar \alpha}(M_\o[\vec{\lb}] \cap V_{\lb+1}) \models \anf{\textit{$X_{\varphi,y}$ has the $\vec{\mathcal{F}}$-Baire property}} \\ 
     \Longleftrightarrow ~ {} ~&  L_{\bar \alpha}(M_\o[\vec{\lb}] \cap V_{\lb+1}) \models \forall y \in V_{\lb+1} ~  \anf{\textit{$X_{\varphi,y}$ has the $\vec{\mathcal{F}}$-Baire property}} \\ 
      \Longleftrightarrow ~ {} ~ & L_{\alpha}(V_{\lb+1}) \models \forall y \in V_{\lb+1} ~  \anf{\textit{$X_{\varphi,y}$ has the $\vec{\mathcal{F}}$-Baire property}}. 
 \end{split}
\end{equation*}

These computation show that every set of  the form $X_{\varphi,y}$ with $y\in V_{\lambda+1}$ has the $\vec{\mathcal{F}}$-Baire property. 
\end{proof}


\section{Open questions}

We close this paper by stating two questions raised by the above results. As mentioned in the introduction, our results suggest that large cardinals assumptions can be studied through the provable validity of Perfect Set Theorems for simply definable sets at singular cardinals of countable cofinality.  
 In particular, our results suggest that the existence of an I2-embedding with critical sequence $\vec{\lambda}$ naturally corresponds to the validity of a Perfect Set Theorem for subsets of $C(\vec{\lambda})$ that are definable by $\Sigma_1$-formulas with parameters in $V_\lambda\cup\{\vec{\lambda}\}$, where $\lambda$ is the supremum of the sequence $\vec{\lambda}$. 
 We therefore ask if the conclusion of Theorem \ref{thm:Main1} can also be derived from substantially weaker large cardinal assumptions: 
 
 \begin{question}
  Let $\vec{\lambda}=\seq{\lambda_n}{n<\omega}$ be a strictly increasing sequence of cardinals with supremum $\lambda$ such that $\lambda_n$ is a ${<}\lambda$-supercompact cardinal for all $n<\omega$. 
   If $X$ is a subset of $\POT{\lambda}$ of cardinality greater than $\lambda$ that is definable by a $\Sigma_1$-formula with parameters in $V_\lambda\cup\{\vec{\lambda}\}$, is there  a perfect embedding $\map{\iota}{{}^\omega\lambda}{\POT{\lambda}}$ with $\ran{\iota}\subseteq X$?
   If yes, what about subsets of $\POT{\lambda}$ that are definable by  $\Sigma_1$-formulas with parameters in $V_\lambda\cup\{V_\lambda,\vec{\lambda}\}$?
 \end{question}

  In another direction, we also ask which large cardinal assumptions are necessary to overcome the limitations to the influence of I2-embeddings given by Theorem \ref{thm:NoI2boldface}. 
  Note that, by Theorem \ref{thm:PerfectI0}, an I0-embedding suffices for this task. 
  Remember that an I1-embedding is a non-trivial elementary embedding $\map{j}{V_{\lambda+1}}{V_{\lambda+1}}$.

  \begin{question}
  Let $\map{j}{V_{\lambda+1}}{V_{\lambda+1}}$ be an I1-embedding. 
   If $X$ is a subset of $\POT{\lambda}$ of cardinality greater than $\lambda$ that is definable by a $\Sigma_1$-formula with parameters in $V_{\lambda+1}$, is there  a perfect embedding $\map{\iota}{{}^\omega\lambda}{\POT{\lambda}}$ with $\ran{\iota}\subseteq X$? 
   If yes, what about subsets of $\POT{\lambda}$ that are definable over $V_\lambda$ by $\Sigma^1_n$-formulas with parameters in $V_{\lambda+1}$?  
 \end{question}
 

\subsection*{Acknowledgements} 
The authors are thankful to the anonymous referees for the careful reading of the manuscript and several helpful comments. 


\subsection*{Funding} 
 The first and second author partially supported by the Italian PRIN 2017 Grant ``Mathematical Logic: models, sets, computability'', the first author is further partially supported by the Italian PRIN 2022 Grant ``Models, Sets and Classifications''. 
The third author  gratefully acknowledges support from the Deutsche Forschungsgemeinschaft (Project number 522490605) and the Spanish Government under grant EUR2022-134032.


\bibliographystyle{alpha}
\bibliography{biblio}
	
\end{document}